\documentclass[12pt,twoside]{article}
\usepackage{amsmath,amssymb,amsthm,mathrsfs,color,times,textcomp,sectsty,verbatim}
\usepackage{xcolor}
\usepackage[colorlinks=true]{hyperref}
\hypersetup{urlcolor=blue, citecolor=red, linkcolor=blue}
\usepackage[colorinlistoftodos]{todonotes}
\numberwithin{equation}{section}
\theoremstyle{plain}
\newtheorem{theorem}{Theorem}[section]
\newtheorem{proposition}[theorem]{Proposition}
\newtheorem{lemma}[theorem]{Lemma}
\newtheorem*{conjecture}{Conjecture}

\theoremstyle{definition}
\newtheorem{definition}[theorem]{Definition}
\newtheorem*{convention}{\textbf{Convention}}
\newtheorem{remark}[theorem]{Remark}

 \def \no{\nonumber}
\def \pa{\partial}
\def\q{\mathcal{Q}_{a,b}^q}

\def\e{\epsilon}

\def\R{\mathbb{R}}
\def\Rn{{\mathbb{R}}^n_+}
\def\d{\partial}

\def\a{\alpha}
\def\b{\beta}

\def\crit {\frac{2n}{n-2}}

\def\ba{\begin{align}}
\def\ea{\end{align}}
\def\bp{\begin{proof}}
\def\ep{\end{proof}}

\def\ubar{\bar{U}_{(x_0,\e)}}

\def\func:u{\bar{u}_{(x_0,\e)}}

\def\U{W_{\epsilon}}

\def\w{\psi}
 \allowdisplaybreaks
 
\begin{document}

\title{\Large \bf
Existence of conformal metrics with constant scalar curvature and constant boundary mean curvature on compact manifolds}

\author{Xuezhang  Chen\thanks{X. Chen: xuezhangchen@nju.edu.cn;  $^\dag$L. Sun: lsun@math.jhu.edu.}~ and Liming Sun$^\dag$\\
 \small
$^\ast$Department of Mathematics \& IMS, Nanjing University, Nanjing
210093, P. R. China\\
\small
$^\dag$Department of Mathematics, Rutgers University, \\
\small
110 Frenlinghuysen Road, Piscataway NJ 08854, USA\\
\small 
$^\dag$Department of Mathematics, Johns Hopkins University, Baltimore, Maryland 21218, USA
}

\date{}

\maketitle

\begin{abstract}
We study the problem of deforming a Riemannian metric to a conformal one with nonzero constant scalar curvature and nonzero constant boundary mean curvature on a compact manifold of dimension $n\geq 3$. We prove the existence of such conformal metrics in the cases of $n=6,7$ or the manifold is spin and some other remaining ones left by Escobar. Furthermore, in the positive Yamabe constant case, by normalizing the scalar curvature to be $1$, there exists a sequence of conformal metrics such that their constant boundary mean curvatures go to $+\infty$.

{{\bf $\mathbf{2010}$ MSC:} Primary 53C21, 35J65; Secondary 58J05,35J20.}

\end{abstract}

%\listoftodos

\section{Introduction}
Analogous to the Yamabe problem, a very natural question on a compact manifold with boundary is, for dimension $n \geq 3$, whether it is possible to deform any Riemannian metric  to a conformal one with constant scalar curvature and constant boundary mean curvature. When studying the above problem, we benefited much from the series of papers on the Yamabe problem by Yamabe, Trudinger, Aubin and Schoen. Readers are referred to Lee and Park \cite{Lee-Parker}, Aubin \cite{aubin_book}  for a survey on the Yamabe problem, see also Bahri and Brezis \cite{Bahri-Brezis}, Bahri \cite{Bahri}  for the works on this problem and related ones.

Let $(M,g_0)$ be a smooth compact Riemannian manifold of dimension $n\geq 3$  with boundary $\partial M$. The problem is equivalent to finding a positive solution to the following PDE:
\begin{align}\label{prob:smc-2constants}
\left\{\begin{array}{ll}
\displaystyle-\frac{4(n-1)}{n-2}\Delta_{g_0}u+R_{g_0}u=c_1 u^{\frac{n+2}{n-2}},&\hspace{2mm}\mbox{ in }M,\\
\displaystyle\frac{2}{n-2}\frac{\partial u}{\partial \nu_{g_0}}+h_{g_0}u=c_2 u^{\frac{n}{n-2}},&\hspace{2mm}\mbox{ on }\partial M,
\end{array}
\right.
\end{align}
where $c_1,c_2\in \mathbb{R}$, $R_{g_0}$ denotes the scalar curvature, $\nu_{g_0}$ denotes the outward unit normal on $\pa M$ and $h_{g_0}$ denotes  the mean curvature of $\pa M$ with respect to $\nu_{g_0}$ (balls in $\mathbb{R}^n$ have positive boundary mean curvature). When $c_1=0$ and $c_2 \in \mathbb{R}$, we refer the scalar-flat metrics of constant mean curvature problem to \cite{escobar1,marques1,marques2,Brendle1,ChenSophie,almaraz5,almaraz1,ChenHo}. When $c_1 \in \mathbb{R}, c_2=0$, we refer the Yamabe problem with minimal boundary to \cite{escobar4,Brendle-Chen,Brendle1,ChenHo,Almaraz-Sun}. When $c_1,c_2\neq 0$, problem \eqref{prob:smc-2constants} is also called constant scalar curvature and constant mean curvature problem. Escobar initiated the investigation of this problem in \cite{escobar5, escobar3}.  In \cite{han-li1,han-li2}, Z. C. Han and Y. Y. Li proposed the following conjecture:
\begin{conjecture}[\textbf {Han-Li}] \textit{If $Y(M,\pa M)>0$, then problem \eqref{prob:smc-2constants} is solvable for any positive constant $c_1$ and any $c_2 \in \mathbb{R}$.}
\end{conjecture}
They proved that the conjecture is true when one of the following assumptions is fulfilled:
\begin{enumerate}
\item[(a)]$n \geq 5$ and $\pa M$ admits at least one non-umbilic point (see \cite{han-li1});
\item[(b)]$n\geq 3$ and $(M,g_0)$ is locally conformally flat with umbilic boundary $\pa M$ (see \cite{han-li2}).
\end{enumerate}

Before presenting our results, we need to introduce natural conformal invariants. The (generalized) Yamabe constant $Y(M,\partial M)$ is defined as
\begin{equation}\label{Yamabe_constant}
Y(M,\partial M):=\inf_{g\in[g_0]}\frac{\int_MR_gd\mu_g+2(n-1)\int_{\partial M}h_gd\sigma_g}{(\int_{M}d\mu_g)^{\frac{n-2}{n}}}.
\end{equation}
Similarly, we define (see \cite{escobar1})
\begin{equation*}
Q(M,\partial M):=\inf_{g\in[g_0]}\frac{\int_MR_gd\mu_g+2(n-1)\int_{\partial M}h_gd\sigma_g}{(\int_{\pa M} d\sigma_g)^{\frac{n-2}{n-1}}}.
\end{equation*}
If $Y(M,\pa M)>(=) 0$, then there exists a conformal metric of $g_0$ with zero scalar curvature in $M$ and positive (zero) mean curvature on $\pa M$.\footnote{Since $Y(M,\pa M)>(=) 0$, it follows from \cite[Lemma 1.1]{escobar4} that there exists $g_1 \in [g_0]$ such that $R_{g_1}>(=) 0$ and $h_{g_1}=0$. Let $\varphi$ be a smooth positive minimizer of $\{\int_M(\frac{4(n-1)}{n-2}|\nabla \psi|_{g_1}^2+R_{g_1}\psi^2)d\mu_{g_1}; \psi\in H^1(M,g_1),\int_{\pa M}\psi^2 d\sigma_{g_1}=1\}$, then $\varphi^{4/(n-2)}g_1$ is the desired conformal metric.} In fact, $Y(M,\pa M)>0$ if and only if $Q(M,\pa M)>0$. However, it was first pointed out by Zhiren Jin (see \cite{escobar2}) that $Q(M,\partial M)$ could be $-\infty$,  meanwhile $Y(M,\partial M)>-\infty$. 

We remark that problem \eqref{prob:smc-2constants} is variational. The total scalar curvature plus total mean curvature functional is given by
\begin{equation}\label{smc_energy}
E[u]=\int_M (\tfrac{4(n-1)}{n-2}|\nabla u|_{g_0}^2 +R_{g_0}u^2) d\mu_{g_0}+2(n-1)\int_{\partial M}h_{g_0}u^2 d\sigma_{g_0}.
\end{equation}
Given any $a,b>0$, we define a conformal invariant by
\begin{align*}
Y_{a,b}(M,\pa M)=&\inf_{g\in[g_0]}\frac{\int_{M}R_gd\mu_g+2(n-1)\int_{\partial M}h_gd\sigma_g}{a\left(\int_{M}d\mu_g\right)^{\frac{n-2}{n}}+2(n-1)b\left(\int_{\partial M}d\sigma_g\right)^{\frac{n-2}{n-1}}}\\
=&\inf_{0\not \equiv u\in H^1(M,g_0)}\mathcal{Q}_{a,b}[u],
\end{align*}
where
$$\mathcal{Q}_{a,b}[u]=\frac{E[u]}{a\left(\int_{M}|u|^{\frac{2n}{n-2}}d\mu_{g_0}\right)^{\frac{n-2}{n}}+2(n-1)b\left(\int_{\partial M}|u|^{\frac{2(n-1)}{n-2}}d\sigma_{g_0}\right)^{\frac{n-2}{n-1}}}.$$
The next theorem gives a criterion for the existence of a minimizer for $Y_{a,b}(M,\pa M)$.

\begin{theorem}\label{Thm:minimizers}
Suppose $Y_{a,b}(M,\pa M)< Y_{a,b}(S^n_+,S^{n-1})$ for some given $a,b>0$, then $Y_{a,b}(M,\pa M)$ can be achieved by a smooth positive minimizer.
\end{theorem}

In \cite{Araujo1,Araujo2} Araujo also gave some characterization of critical points (including minimizers) of $E[u]$ under Escobar's non-homogeneous constraint (see \cite{escobar3}).

In order to apply Theorem \ref{Thm:minimizers} in the case of $Y(M,\pa M)>0$, we need to construct a global test function $\ubar$ as a small perturbation of a bubble function $\U$ with $x_0 \in \pa M$ and small $\e>0$, such that $\mathcal{Q}_{a,b}[\ubar]< Y_{a,b}(S_+^n,S^{n-1})$. We would like to mention some developments on the technique of constructing test functions in very closely related works. In dimension $n\geq 6$, Brendle \cite{Brendle2} initiated this technique of constructing test functions through his study of the Yamabe flow. Subsequently Brendle and S. Chen \cite{Brendle-Chen} developed it to study the Yamabe problem with umbilic minimal boundary (i.e. $c_1\in \mathbb{R},c_2=0$ and umbilc boundary). Not long after that S. Chen \cite{ChenSophie} adapted the same technique to scalar-flat and constant mean curvature problem with umbilic boundary (i.e. $c_1=0,c_2\in \mathbb{R}$ and umbilc boundary). One of the key ingredients in Almaraz \cite{almaraz5} and Almaraz-L. Sun \cite{Almaraz-Sun} is to extend such a technique to the case of the boundary $\pa M$ having at least one non-umbilic point and the case of lower dimension $3 \leq n \leq 5$. The correction term $\w$ in our test function (see \eqref{eq:test_fcn}) comes from the linearized equations of scalar curvature and mean curvature at a round metric on a spherical cap, which has constant sectional curvature $4$ (see Proposition \ref{prop:linearized_eqs_smc}).

For $n \geq 3$, let $\mathbb{R}_+^n=\{y=(y^1,\cdots,y^n) \in \mathbb{R}^n; y^n>0\}$ denote the half Euclidean space. We shall use a notion of a {\it mass} associated to manifolds with boundary.
\begin{definition}\label{def:asym}
Let $(N, g)$ be a Riemannian manifold with non-compact boundary $\d N$. 
We say that $N$ is {\it{asymptotically flat}} with order $p>0$, if there exist a compact set $N_0\subset N$ and a diffeomorphism $F: N\backslash N_0\to \Rn\backslash \overline{B^+_1(0)}$ such that, in the coordinate chart defined by $F$ (called {\it  asymptotic coordinates} of $N$), there holds
$$
|g_{ij}(y)-\delta_{ij}|+|y||\pa_k g_{ij}(y)|+|y|^2|\pa^2_{kl}g_{ij}(y)|=O(|y|^{-p}),\hbox{~~as~~}|y|\to\infty,
$$
where $i,j,k,l=1,\cdots,n, B_1^+(0)=B_1(0)\cap \mathbb{R}_+^n$.
\end{definition}

If $R_g, h_g$ are integrable on $N$ and $\pa N$ respectively, and the decay order of $(N,g)$ is $p>(n-2)/2$, then the following limit
\begin{align*}
m(g):=\lim_{R\to\infty}\left[
\int\limits_{\{y\in\Rn;\, |y|=R\}}\sum_{i,j=1}^{n}(g_{ij,j}-g_{jj,i})\frac{y^i}{|y|}\,d\sigma
+\int\limits_{\{y\in\mathbb{R}^{n-1};\, |y|=R\}}\sum_{a=1}^{n-1}g_{na}\frac{y^a}{|y|}\,d\sigma\right]
\end{align*}
exists, and we call it the {\it{mass}} of $(N, g)$. Moreover, $m(g)$ is a geometric invariant in the sense that it is independent of asymptotic coordinates. The definition of the mass $m(g)$ was first proposed by Marques. The following positive mass type conjecture was initially given in \cite{almaraz5} and has been verified in \cite[Theorem 1.3]{almaraz-barbosa-lima} under the hypotheses that either $3 \leq n \leq 7$ or if $n\geq 3$ and $N$ is spin.

\begin{conjecture}[\textbf{Positive mass with a non-compact boundary}]
If $(N,g)$ is asymptotically flat with decay order of $p>(n-2)/2$ and $R_g, h_g\geq 0$, then we have $m(g)\geq 0$ and the equality holds if and only if $N$ is isometric to $\Rn$.
\end{conjecture}

For $n\geq 3$, let $d=[(n-2)/2]$. As in \cite{almaraz5}, we define 
\begin{align*}
\mathcal{Z}=
\{x_0\in\pa M;&\limsup_{x\to x_0}d_{g_0}(x,x_0)^{2-d}|W_{g_0}(x)|_{g_0}=0\text{ and }\\
&\limsup_{x\to x_0}d_{g_0}(x,x_0)^{1-d}|\mathring{\pi}_{g_0}(x)|_{g_0}=0\},
\end{align*}
where $W_{g_0}$ denotes the Weyl tensor of $M$, $\pi_{g_0}$ and $\mathring{\pi}_{g_0}$ denote the second fundamental form and its trace-free part by: for any $x \in \pa M$, let $X,Y \in T_x (\pa M)$, then
$$\pi_{g_0}(X,Y)=\langle\nabla_X \nu_{g_0},Y\rangle_{g_0} \quad\hbox{and}\quad \mathring{\pi}_{g_0}(X,Y) =\pi_{g_0}(X,Y)-h_{g_0}\langle X,Y\rangle_{g_0}.$$
Then $\mathcal{Z}$ only depends on the conformal structure of $g_0$, since $W_{g_0}$ and $\mathring{\pi}_{g_0}$ are both pointwise conformal invariants. In particular, $\mathcal{Z}=\pa M$ when $n=3$. 

For $x_0\in \pa M$, let $g_{x_0}\in [g_0]$ be the metric induced by the conformal Fermi coordinates around $x_0$ (see \cite{marques2}). Denote by $G_{x_0}$ the Green's function of conformal Laplacian of $g_{x_0}$ with pole at $x_0$, satisfying the boundary condition $\pa_{\nu_{g_{x_0}}} G_{x_0}-\frac{n-2}{2} h_{g_{x_0}}G_{x_0}=0$ on $\pa M \setminus \{x_0\}$ (see \eqref{def:Green_funct}). Let $\bar g_{x_0}=G_{x_0}^{4/(n-2)}g_{x_0}$. Now we are ready to state our main result.

\begin{theorem}\label{Thm:main} Let $(M,g_0)$ be a smooth compact Riemannian manifold of dimension $n\geq 3$ with boundary. Suppose that $M$ is not conformal to the standard hemisphere $S_+^n$. If $Y(M,\pa M)>0$, assume either $\partial M\backslash \mathcal{Z}\neq \emptyset$ or  $m(\bar{g}_{x_0})>0$ for some $x_0\in \mathcal{Z}$, then
$$Y_{a,b}(M,\pa M)< Y_{a,b}(S^n_+,S^{n-1}).$$
\end{theorem}
We should point out that such assumptions on compact manifolds in Theorem \ref{Thm:main} (or with some minor modifications) have been used in some closely related problems. For instance, Brendle \cite{Brendle2} for the Yamabe flow in dimension $n\geq 6$, S. Chen \cite{ChenSophie} and Almaraz \cite{almaraz5} for $c_1=0, c_2 \in \mathbb{R}$, Brendle-Chen \cite{Brendle-Chen} and Almaraz-L. Sun \cite{Almaraz-Sun} for $c_1 \in \mathbb{R}, c_2=0$.

Recent advances in the above positive mass type theorem (see \cite{almaraz-barbosa-lima,almaraz5,Raulot} etc.) have played an important role in such prescribed conformal curvature problems. As a direct consequence of Theorem \ref{Thm:main} and the positive mass type theorem recently proved in \cite{almaraz-barbosa-lima}, we obtain

\begin{theorem}\label{thm:explicit_assumptions}
Let $(M,g_0)$ be a smooth compact Riemannian manifold of dimension $n\geq 3$ with boundary. Suppose that $M$ is not conformal to the standard hemisphere $S_+^n$ and $Y(M,\pa M)>0$.  Assume that one of the following assumptions is satisfied:
\begin{enumerate}
\item[(i)] $\pa M \setminus \mathcal{Z} \neq \emptyset$;
\item[(ii)] $3 \leq n \leq 7$ or $M$ is spin; 
\item[(iii)] $n\geq 8$ and $(M,g_0)$ is locally conformally flat with umbilic boundary $\pa M$.
\end{enumerate}
Then given any $a,b>0$, there exists at least one smooth positive minimizer $u_{a,b}$ for $Y_{a,b}(M,\pa M)$. Moreover, the conformal metric $u_{a,b}^{4/(n-2)}g_0$, modulo a positive constant multiple, has scalar curvature $1$ and some positive constant boundary mean curvature.
\end{theorem}
When $Y(M,\pa M)>0$, Escobar proved in \cite[Theorem 4.2]{escobar3} the existence of such conformal metrics in Theorem \ref{thm:explicit_assumptions} under one of the following hypotheses:
\begin{enumerate}
\item[(1)] $3 \leq n\leq 5$;
\item[(2)] $\pa M$ has at least a nonumbilic point; 
\item[(3)] $\pa M$ is umbilic and either $M$ is locally conformally flat or the Weyl tensor is not identically zero on $\pa M$.
\end{enumerate} 
Then we generalize the existence results to the cases including $n=6,7$ or $M$ is spin. Some remaining cases left by Escobar are the manifolds  that are not locally conformally flat and $\pa M$ is umbilic, and Weyl tensor vanishes identically on $\pa M$ and $n\geq 6$. Thus our Theorem \ref{thm:explicit_assumptions} also generalizes to this type of manifolds under the assumption $\pa M \backslash\mathcal{Z}\neq \emptyset$. We next prove the compactness of the minimizers for $Y_{a,b}(M, \pa M)$ when $(a,b)$ varies in a compact set $K$ of $\{(a,b); a\geq 0,b\geq 0\}\setminus\{(0,0)\}$. We denote by $\mathcal{M}_{a,b}$ the set of smooth positive minimizers of $Y_{a,b}(M,\pa M)$ with the normalization \eqref{normalization_minimizer}.
\begin{theorem}\label{thm:compactness_minimizers}
Let $K$ and $\mathcal{M}_{a,b}$ be defined above. Suppose $Y_{a,b}(M,\pa M)<Y_{a,b}(S^n_+,S^{n-1})$ for all $(a,b)\in K$, then there exists $C=C(K,g_0)$ such that 
$$C^{-1}\leq u_{a,b}\leq C ,\quad \|u_{a,b}\|_{C^2(M)}\leq C, \quad  \forall~~ u_{a,b} \in \cup_{(a,b) \in K}\mathcal{M}_{a,b}.$$
\end{theorem}

It follows from Theorem \ref{thm:compactness_minimizers} that in terms of normalized conformal metrics having scalar curvature $1$, there exits a sequence of such conformal metrics such that their constant boundary mean curvatures go to $+\infty$. We refer to the end of Section \ref{Sect4} for details. In contrast with our result, the constant mean curvature of such a conformal metric in \cite[Theorem 4.2]{escobar3} only admits a small real number. Indeed, the smallness of $b \in \mathbb{R}$ in a conformal invariant $G_{a,b}(M)$ (see also Section \ref{Sect2}) is very crucial in the proof of \cite{escobar3}.
\begin{remark}
When $Y(M,\pa M)<0$, as a direct consequence of \cite[Theorem 1.1]{ChenHoSun},
there exists a conformal metric such that its scalar curvature equals $-1$ and its boundary mean curvature equals any negative real number.
\end{remark}
\begin{remark}
The assumptions (ii) and (iii) enable us to use the above positive mass type theorem results of \cite{almaraz5} and the appendix of \cite{escobar4}. Due to similar technical reasons as in the study of the Yamabe flow in \cite{Brendle2}, these existence results are reduced to the validity of the above positive mass type conjecture in higher dimension $n \geq 8$.
\end{remark}

In a forthcoming paper \cite{ChenHoSun1} we will adopt a geometric flow with another Escobar's non-homogeneous constraint used in \cite{escobar3} to tackle problem \eqref{prob:smc-2constants}.  In general, such a geometric flow can be used to find some non-minimizer critical points of the associated functional.  More related conformal curvature flows can be found in \cite{Brendle4,Brendle2,Brendle1,ChenHo,ChenHoSun,almaraz5} and the references therein.

The present paper is organized as follows. In Section \ref{Sect2}, we describe some properties of the standard bubble on the boundary and conformal invariant $Y_{a,b}(S^n_+,S^{n-1})$. In section \ref{Sect3},  a procedure of subcritical approximations is set up to prove Theorem \ref{Thm:minimizers}. We present the compactness of  the set $\mathcal{M}_{a,b}$ with various $(a,b)\in K$ in Section \ref{Sect4}. In section \ref{Sect5}, we derive the detailed computations for the linearization of scalar curvature and mean curvature at a round metric on a spherical cap  in Section \ref{Subsect5.1}, which is of independent interest. Finally in Section \ref{Subsect5.2}, we construct these desired test functions required by Theorem \ref{Thm:minimizers} and establish its energy estimates, then Theorem \ref{Thm:main} follows.\\

\noindent{\bf Acknowledgements.} Both authors would like to thank Professor Yan Yan Li for his fruitful discussions and encouragement. This work was carried out during the first author's one-year visit to Rutgers University. He also would like to thank Mathematics Department at Rutgers University for its hospitality and financial support. He was supported by NSFC (Nos. 11771204 and 11201223), A Foundation for the Author of National Excellent Doctoral Dissertation of China (No.201417), Program for New Century Excellent Talents in University (NCET-13-0271), the travel grants from AMS Fan fund and Hwa Ying foundation at Nanjing University. We thank the referee for careful reading and for helpful mathematical comments which improved the exposition.

\section{Preliminaries}\label{Sect2}
Let $T_c$ be a negative real number, it follows from the classification theorem in \cite{li-zhu} that all nonnegative $C^2$ solutions to the following PDE
\begin{align}\label{prob:half-space}
\begin{cases}
\displaystyle-\Delta v=n(n-2) v^{\frac{n+2}{n-2}},\quad &\text{~~in~~}\R^n_+,\\
\displaystyle \frac{\pa v}{\pa y^n}=(n-2)T_c v^{\frac{n}{n-2}},\quad &\text{~~on~~} \R^{n-1},
\end{cases}
\end{align}
must be either $v\equiv 0$ or $v(y)=W_\e(y)$ (up to translations in the variables $y^1,\cdots,y^{n-1}$), where
\begin{equation}\label{eq:bdry_bubble}
\U(y)=\left(\frac{\e}{\e^2+|y-T_c\e \mathbf{e}_n|^2}\right)^{\frac{n-2}{2}}, ~\forall~\e>0.
\end{equation}
Here $\mathbf{e}_n$ is the unit direction vector in the $n$-th coordinate.
Also $W_\e$ are the extremal functions of  the best Sobolev constant  
$$Y_{a,b}(\mathbb{R}^n_+,\mathbb{R}^{n-1})=\inf_{0\not\equiv\varphi \in C_c^\infty(\overline{\mathbb{R}_+^n})}\frac{\frac{4(n-1)}{n-2}\int_{\mathbb{R}_+^n} |\nabla \varphi|^2 dx}{a \left(\int_{\mathbb{R}_+^n}|\varphi|^{\frac{2n}{n-2}}dx\right)^{\frac{n-2}{n}}+2(n-1)b \left(\int_{\mathbb{R}^{n-1}}|\varphi|^{\frac{2(n-1)}{n-2}}d\sigma\right)^{\frac{n-2}{n-1}}}$$
with some $a,b>0$ (see \cite[Theorem 3.3]{escobar5} or Lemma \ref{lem:Sobolev_trace_mfld} below).

By definition of ~$Y_{a,b}(M,\pa M)$, any positive minimizers in $H^1(M,\pa M)$ of $Y_{a,b}(M,\pa M)$ satisfy
\begin{align}\label{eq:critical_exponent}
\begin{cases}
\displaystyle-\frac{4(n-1)}{n-2}\Delta_{g_0}u+R_{g_0}u=\mu(M) a \left(\int_Mu^\frac{2n}{n-2}d\mu_{g_0}\right)^{-\frac{2}{n}}u^{\frac{n+2}{n-2}}&\hbox{~~in~~} M,\\
\displaystyle\frac{2}{n-2}\frac{\partial u}{\partial{\nu_{g_0}}}+h_{g_0}u=\mu(M) b\left(\int_{\partial M}u^{\frac{2(n-1)}{n-2}}d\sigma_{g_0}\right)^{-\frac{1}{n-1}}u^{\frac{n}{n-2}}&\hbox{~~on~~}\partial M,
\end{cases}
\end{align}
where $\mu(M)=Y_{a,b}(M,\pa M)$. It is not hard to see that the Euler-Lagrange equation for $Y_{a,b}(\mathbb{R}^n_+,g_{\mathbb{R}^n})$ is equivalent to finding positive solutions of the following PDE:
\begin{align}\label{eq:Rncase}
\begin{cases}
\displaystyle-\frac{4(n-1)}{n-2}\Delta u=\mu a \left(\int_{\mathbb{R}_+^n}u^\frac{2n}{n-2}dx\right)^{-\frac{2}{n}}u^{\frac{n+2}{n-2}}\qquad\hbox{ in } \mathbb{R}^n_+,\\
-\displaystyle\frac{2}{n-2}\frac{\partial u}{\partial{y^n}}=\mu b\left(\int_{\mathbb{R}^{n-1}}u^{\frac{2(n-1)}{n-2}}d\sigma\right)^{-\frac{1}{n-1}}u^{\frac{n}{n-2}}\quad\text{ on } \mathbb{R}^{n-1},
\end{cases}
\end{align}
 where $\mu=\mu(\mathbb{R}_+^n)=Y_{a,b}(\mathbb{R}^n_+,\mathbb{R}^{n-1})$. A simple but vital observation is that if $u$ is a smooth positive solution to problem \eqref{eq:Rncase}, so is $c_\ast u$ for all $c_\ast \in \mathbb{R}_+$.
Hence all positive solutions to problem \eqref{eq:Rncase} are in the form of, up to dilations and translations in the variables $y^1,\cdots,y^{n-1},$
$$c_\ast  \left(\frac{1}{1+|y-T_c \mathbf{e}_n|^2}\right)^{\frac{n-2}{2}}$$
for all $c_\ast>0$ and some $T_c<0$ depending on $n,a,b$. We choose $c_\ast=1$ hereafter, namely, for this fixed $T_c<0$, the associated function
$$W(y)=\left(\frac{1}{1+|y-T_c \mathbf{e}_n|^2}\right)^{\frac{n-2}{2}}$$
is a positive solution to both PDEs \eqref{prob:half-space} and \eqref{eq:Rncase}.

Next we give a geometric interpretation for the standard boundary bubble function $W$. Denote by a mapping $\pi: S^n(T_c\mathbf{e}_n) \setminus\{T_c \mathbf{e}_n+\mathbf{e}_{n+1}\} \to \{\xi+T_c \mathbf{e}_n \in \mathbb{R}^{n+1};\xi^{n+1}=0\}\simeq \mathbb{R}^n$ the stereographic projection from the unit sphere $S^n(T_c\mathbf{e}_n)$ in $\mathbb{R}^{n+1}$ centered at $T_c \mathbf{e}_n$. Then for $y \in \mathbb{R}_+^n$, we set $\xi=\pi^{-1}(y) \in S^n$, namely (see also  \cite[(3.1) on page 831]{han-li1})
\begin{align*}
\left\{\begin{array}{ll}
\displaystyle \xi^a=\frac{2y^a}{1+|y-T_c\mathbf{e}_n|^2}, \quad \hbox{for~~} 1 \leq a \leq n-1,\\
\displaystyle \xi^n=\frac{2(y^n-T_c)}{1+|y-T_c\mathbf{e}_n|^2},\\
\displaystyle \xi^{n+1}=\frac{|y-T_c\mathbf{e}_n|^2-1}{1+|y-T_c\mathbf{e}_n|^2}.
\end{array}
\right.
\end{align*}
Let $\Sigma$ be a spherical cap (see Figure \ref{fig:stereo_graphic}) equipped with a round metric $\frac{1}{4}g_{S^n}$, where $g_{S^n}$ is the standard metric of the unit sphere $S^n(T_c\mathbf{e}_n)$. Then a direct computation shows
\begin{equation*}
\frac{1}{4}(\pi^{-1})^\ast(g_{S^n})=\left(\frac{1}{1+|y-T_c\mathbf{e}_n|^2}\right)^2 g_{\mathbb{R}^n}=W(y)^{\frac{4}{n-2}}g_{\mathbb{R}^n}.
\end{equation*}
From this and the fact that $Y_{a,b}(\Sigma,\pa \Sigma)$ is a conformal invariant, we conclude that
$$Y_{a,b}(\mathbb{R}^n_+,\mathbb{R}^{n-1})=Y_{a,b}(\Sigma,\partial \Sigma)=Y_{a,b}(S^n_+,S^{n-1}).$$
To simplify the notation, we will sometimes use $Y_{a,b}(\mathbb{R}^n_+,\mathbb{R}^{n-1})$ with a slightly relaxed condition of $a,b$ such that $a,b\geq 0, a^2+b^2>0$.
\begin{figure}[ht]
\vskip -0.2 in
\centering
\includegraphics[width=5.5in]{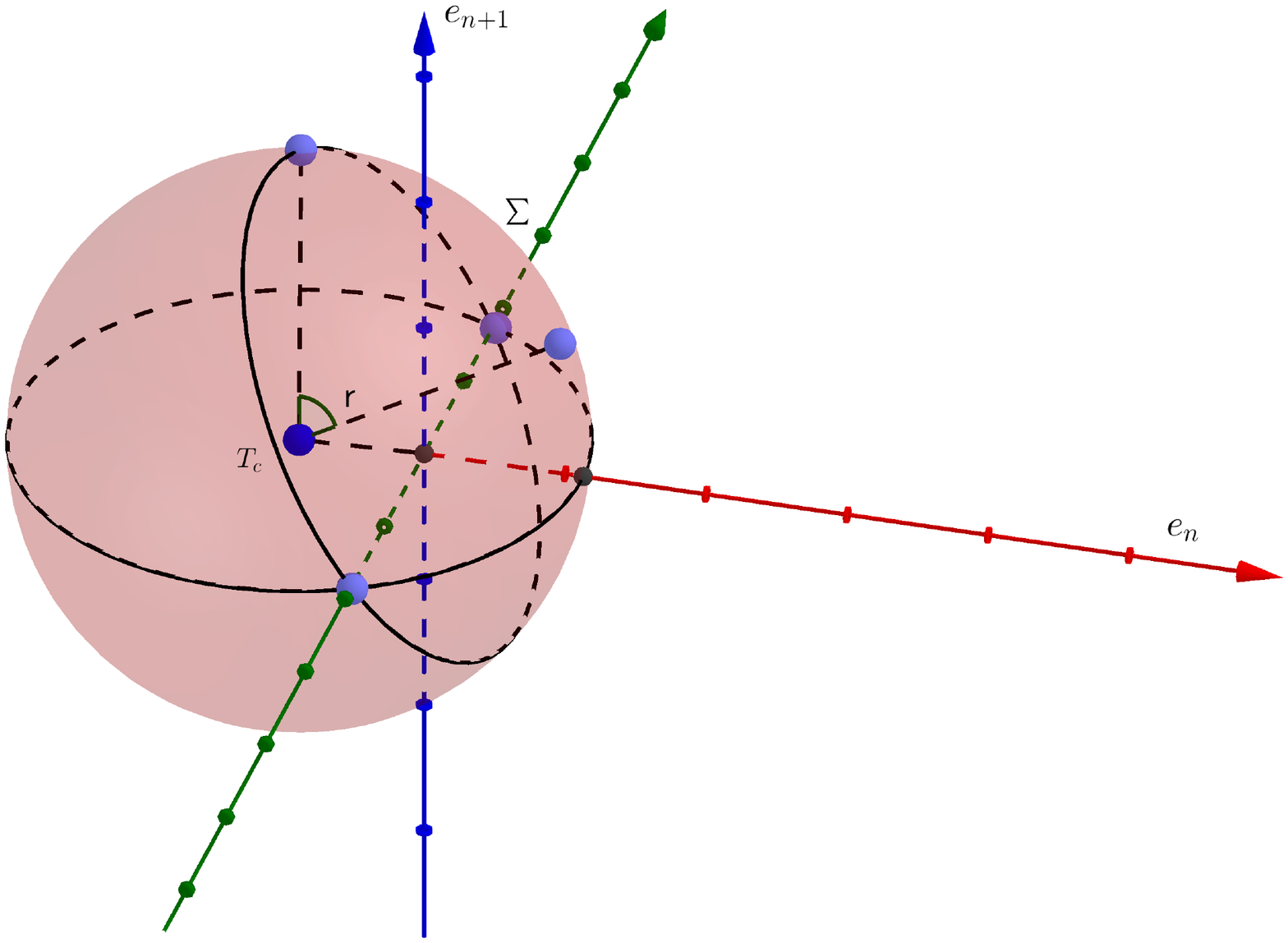}
\vskip -0.4in
\caption{Stereographic projection for $T_c<0$}\label{fig:stereo_graphic}
\end{figure}

Let  $\omega_{n-1}$ denote the volume of the standard unit sphere in $\mathbb{R}^n$. Define
\begin{align*}
A=\int_{\mathbb{R}^n_+}W(y)^{\frac{2n}{n-2}}dy\quad\hbox{and}\quad B=\int_{\mathbb{R}^{n-1}}W(y)^{\frac{2(n-1)}{n-2}}d\sigma.
\end{align*}
Notice that $A,B$ only depend on $n, T_c$.  Using \eqref{prob:half-space} we get
 \begin{align}\label{energy:bubble}
 \int_{\mathbb{R}^n_+}|\nabla W(y)|^2dy=n(n-2)A-(n-2)T_cB.
 \end{align}
% For such $\U$, \eqref{eq:crit_point_eq} is equivalent to 
% \begin{align*}
% \begin{cases}
% \displaystyle-\frac{4(n-1)}{n-2}\Delta u=\lambda a A^{-\frac{2}{n}}u^{\frac{n+2}{n-2}}\\
% \displaystyle-\frac{2}{n-2}\frac{\partial u}{\partial{y_n}}=\lambda bB^{-\frac{1}{n-1}}u^{\frac{n}{n-2}}
% \end{cases}
% \end{align*}
% This is equivalent to 
% \begin{align}
% \begin{cases}
% -\Delta u=\lambda \frac{n-2}{4(n-1)}a A^{-\frac{2}{n}}u^{\frac{n+2}{n-2}}\\
% -\frac{\partial u}{\partial{y_n}}=\lambda\frac{n-2}{2} bB^{-\frac{1}{n-1}}u^{\frac{n}{n-2}}
% \end{cases}
% \end{align}
Recall that, from \cite[Theorem 3.3]{escobar5} that $Y_{a,b}(\mathbb{R}^n_+,\mathbb{R}^{n-1})$ can be achieved by $W$ with some $T_c$ (up to dilations and translations in variables $y^1,\cdots,y^{n-1}$) modulo a positive constant multiple. Comparing \eqref{eq:Rncase} and \eqref{prob:half-space}, as well as the above comments, we have
\begin{align*}
\mu \frac{n-2}{4(n-1)}aA^{-\frac{2}{n}}=n(n-2),\quad \mu \frac{n-2}{2}b B^{-\frac{1}{n-1}}=-(n-2)T_c,
\end{align*}
whence
$$-aA^{-\frac{2}{n}}T_c=2n(n-1)bB^{-\frac{1}{n-1}}.$$
Indeed we will establish that each pair of $a,b>0$ corresponds to a unique $T_c$ satisfying the above identity.
\begin{lemma} \label{lem:Y_{a,b}halfspace}
Given any $a,b>0$, there exists a unique $T_c\in (-\infty,0)$ such that
\begin{align}\label{eq:AB_T_c}
-aA^{-\frac{2}{n}}T_c=2n(n-1)bB^{-\frac{1}{n-1}}.
\end{align}
In particular, $T_c$ is a continuous function of $(a,b) \in \mathbb{R}_+\times \mathbb{R}_+$.
Moreover, for such a $W$ satisfying \eqref{eq:Rncase} with the above unique $T_c$\,, there holds
\begin{align*}
Y_{a,b}(\mathbb{R}^n_+,\mathbb{R}^{n-1})= 4n(n-1)a^{-1}A^{\frac{2}{n}}=-2T_cb^{-1}B^{\frac{1}{n-1}}.
\end{align*}
\end{lemma}
\begin{proof}
Let $\cos r=\frac{-T_c}{\sqrt{1+T_c^2}}, r \in (0,\frac{\pi}{2})$, then $A$ and $B$ turn to
\begin{equation}\label{eq:A_B}
2^n A(r)=\omega_{n-1}\int_0^r(\sin \tau)^{n-1}d\tau,~~2^{n-1}B(r)=\omega_{n-1}(\sin r)^{n-1}.
\end{equation}
Then equation \eqref{eq:AB_T_c} is equivalent to finding some $r \in (0,\frac{\pi}{2})$ such that
$$f(r):=2n(n-1)b A^{\frac{2}{n}}B^{-\frac{1}{n-1}}-a\cot r=0.$$
First it is easy to verify that
$$\lim_{r \searrow 0} f(r)=-\infty \hbox{~~and~~} \lim_{r\nearrow \frac{\pi}{2}}f(r)=\mathrm{constant}>0.$$
Next we claim that $f(r)$ is increasing in $(0,\frac{\pi}{2})$. To see this, we have
\begin{align*}
&\frac{d}{d r}\log (B^{-\frac{1}{n-1}}A^{\frac{2}{n}})=\frac{2}{n}\frac{A'}{A}-\frac{1}{n-1}\frac{B'}{B}\\
=&\frac{1}{\sin r \int_0^r(\sin\tau)^{n-1}d\tau}\left[\frac{2}{n} (\sin r)^n-\cos r \int_0^r(\sin\tau)^{n-1}d\tau\right].
\end{align*}
Observe that
$$\cos r \int_0^r(\sin\tau)^{n-1}d\tau\leq \int_0^r(\sin\tau)^{n-1} \cos \tau d\tau=\frac{1}{n}(\sin r)^n.$$
This implies $(B^{-\frac{1}{n-1}}A^{\frac{2}{n}})(r)$ is increasing in $(0,\frac{\pi}{2})$, as well as is $f(r)$. Hence we conclude that there exists a unique $r \in (0,\frac{\pi}{2})$ such that $f(r)=0$, namely there exists a unique $T_c<0$ satisfying \eqref{eq:AB_T_c}. 

By \cite[Theorem 3.3]{escobar5}, \eqref{energy:bubble} and \eqref{eq:AB_T_c}, we get
\begin{align}\label{conf_invar_halfspace}
Y_{a,b}(\mathbb{R}^n_+,\mathbb{R}^{n-1})=& \frac{\frac{4(n-1)}{n-2}\int_{\mathbb{R}^n_+}|\nabla W|^2dy}{aA^{\frac{n-2}{n}}+2(n-1)bB^{\frac{n-2}{n-1}}}\no\\
=&\frac{4(n-1)}{n-2}\frac{n(n-2)A-(n-2)T_cB}{aA^{\frac{n-2}{n}}+2(n-1)bB^{\frac{n-2}{n-1}}}\no\\
=&4n(n-1)A^{\frac{2}{n}}a^{-1}=-2B^{\frac{1}{n-1}}T_cb^{-1}.
\end{align}
In terms of the variable $T_c$, it follows from \eqref{eq:A_B} that $A(T_c)$ is increasing in $(-\infty,0)$. One may regard $T_c$ as a function of $(a,b)$. Indeed one can show that $Y_{a,b}(\mathbb{R}_+^n,\mathbb{R}^{n-1})$ is continuous in $(a,b) \in \mathbb{R}_+\times \mathbb{R}_+$ (see Proposition \ref{prop:continuity_conformal_invariant} below). From this and the third identity in \eqref{conf_invar_halfspace}, we get $A$ is a continuous function of $(a,b)$.  Hence we conclude that $T_c$ is a continuous function in $(a,b)$.
\end{proof}

From now on, we fix $T_c<0$ as the unique one in Lemma \ref{lem:Y_{a,b}halfspace} without otherwise stated. 
In \cite{escobar3}, Escobar introduced a conformal invariant by $G_{a,b}=\inf\{E[u]; u\in C_{a,b}\}$, where $a>0, b \in \mathbb{R}$ and
$$C_{a,b}=\left\{u\in C^1(\bar{M}); a\int_M |u|^{\frac{2n}{n-2}}d\mu_{g_{0}}+b\int_{\partial M}|u|^{\frac{2(n-1)}{n-2}}d\sigma_{g_0}=1\right\}.$$
He established that $G_{a,b}(M)\leq G_{a,b}(\mathbb{R}^n_+)$ holds for any compact Riemannian manifold with boundary. By similarly constructing a local test function as a perturbation of $W_\e$ under the Fermi coordinates around a boundary point, one can mimick the proof of \cite[Proposition 3.1]{escobar3} to show $Y_{a,b}(M,\pa M)\leq Y_{a,b}(\mathbb{R}^n_+,\mathbb{R}^{n-1})$. Since it is more or less standard to the experts in this field, we omit the details here.

\section{A criterion of the existence of minimizers}\label{Sect3}
The purpose of this section is to establish Theorem \ref{Thm:minimizers}.  We adopt the method of subcritical approximations to realize it. For $1< q\leq \frac{n+2}{n-2}$, we define
\begin{align*}
\q[u]=\frac{E[u]}{a\left(\int_M|u|^{q+1}d\mu_{g_0}\right)^{\frac{2}{q+1}}+2(n-1)b\left(\int_{\partial M}|u|^{\frac{q+3}{2}}d\sigma_{g_0}\right)^{\frac{4}{q+3}}}
\end{align*}
for any $u \in H^1(M,g_0)$. Notice that $\q[u]$ always has a lower bound when $Y(M,\pa M)\geq 0$, we set 
$$\mu_q=\inf_{0\not \equiv u\in H^1(M,g_0)}\q[u].$$
For brevity, we use $\mu_{(n+2)/(n-2)}=Y_{a,b}(M,\pa M)$ and $\mathcal{Q}_{a,b}^{(n+2)/(n-2)}[u]=\mathcal{Q}_{a,b}[u]$.  
\begin{lemma}\label{lem:subcritical_invariants}
Given $a,b>0$, there holds $\limsup_{q\nearrow \frac{n+2}{n-2}}\mu_q\leq Y_{a,b}(M,\pa M).$ Moreover, if $Y(M,\pa M)\geq 0$, there holds $\lim_{q \nearrow \frac{n+2}{n-2}}\mu_q=Y_{a,b}(M,\pa M)$.
\end{lemma}
\begin{proof}
For any $\e>0$, there exists $\bar u>0$ such that $\mathcal{Q}_{a,b}[\bar u]\leq Y_{a,b}(M,\pa M)+\e$. For each $\bar u$, there holds $\lim_{q \nearrow \frac{n+2}{n-2}}\q[\bar u]=\mathcal{Q}_{a,b}[\bar u]$. Then we have
$$\limsup_{q \nearrow \frac{n+2}{n-2}}\mu_q\leq \limsup_{q \nearrow \frac{n+2}{n-2}}\q[\bar u]\leq Y_{a,b}(M,\pa M)+\e,$$
which gives the first assertion. If $Y(M,\pa M)\geq 0$, then $E[u]\geq 0$ for any $u \in H^1(M,g_0)$. Notice that 
\begin{align*}
\mathcal{Q}_{a,b}[u]=\q[u]\frac{a\left(\int_{M}|u|^{\frac{q+2}{2}}d\mu_{g_0}\right)^{\frac{2}{q+1}}+2(n-1)b\left(\int_{\partial M}|u|^{\frac{q+3}{2}}d\sigma_{g_0}\right)^{\frac{4}{q+3}}}{a\left(\int_{M}|u|^{\frac{2n}{n-2}}d\mu_{g_0}\right)^{\frac{n-2}{n}}+2(n-1)b\left(\int_{\partial M}|u|^{\frac{2(n-1)}{n-2}}d\sigma_{g_0}\right)^{\frac{n-2}{n-1}}}.
\end{align*}
Hence the second assertion follows by H\"older's inequality and letting $q \nearrow \frac{n+2}{n-2}$.
\end{proof}
\begin{remark}
We point out that there also holds $\lim_{q \nearrow \frac{n+2}{n-2}}\mu_q=Y_{a,b}(M,\pa M)$ when $Q(M,\pa M)$ is a negative real number (see \cite[Remark 7.1]{ChenHoSun}).
\end{remark}
Again thanks to  \cite{ChenHoSun}, it is enough to prove Theorem \ref{Thm:minimizers} when $Y(M,\pa M)\geq 0$.
\begin{lemma}\label{lem:est_H^1_fcns}
Let $(M,g_0)$ be a smooth compact Riemannian manifold of dimension $n \geq 3$ with boundary. Let $2\leq p<\frac{2(n-1)}{n-2}$, then given any $\e>0$, there exists $C=C(n,M,g_0)>0$ such that
$$\left(\int_{\pa M}|\varphi|^p d\sigma_{g_0}\right)^{\frac{2}{p}}\leq \e \int_M |\nabla \varphi|_{g_0}^2 d\mu_{g_0}+\frac{C}{\e}\int_M \varphi^2d\mu_{g_0} $$
for any $\varphi\in H^1(M,g_0)$.
\end{lemma}
\begin{proof}
By negation, there exist some $\e_0>0$ and $\{\varphi_j;j \in \mathbb{N}\}\subset H^1(M,g_0)$ such that
$$1=\left(\int_{\pa M}|\varphi_j|^p d\sigma_{g_0}\right)^{\frac{2}{p}}> \e_0 \int_M |\nabla \varphi_j|_{g_0}^2 d\mu_{g_0}+\frac{j}{\epsilon_0}\int_M \varphi_j^2d\mu_{g_0}.$$
From this, $\{\varphi_j\}$ is uniformly bounded in $H^1(M,g_0)$ and $\int_M \varphi_j^2d\mu_{g_0} \to 0$ as $j \to \infty$. Then up to a subsequence, $\varphi_j \rightharpoonup \varphi$ weakly in $H^1(M,g_0), \varphi_j \to \varphi$ strongly in $L^2(M,g_0)$ and $L^p(\pa M,g_0)$ as $j \to \infty$. Notice that $\varphi_j \to 0$ in $L^2(M,g_0)$ as $j \to \infty$. Thus we obtain $\varphi=0$ a.e. in $M$, which contradicts $\int_{\pa M}|\varphi|^p d\sigma_{g_0}=\lim_{j \to \infty}\int_{\pa M} |\varphi_j|^p d\sigma_{g_0}=1$.
\end{proof}

\begin{lemma}\label{lem:Sobolev_trace_mfld}
Let $(M,g_0)$ be a smooth compact Riemannian manifold of dimension $n\geq 3$ with boundary. Given $a,b>0$, then
\begin{enumerate}
\item [(i)] Let $\varphi\in C^{\infty}_c(\overline{\mathbb{R}^n_+})$, there holds
\begin{align*}
&a\left(\int_{\mathbb{R}_+^n} |\varphi|^{\frac{2n}{n-2}}dy\right)^{\frac{n-2}{n}}+2(n-1)b\left(\int_{\mathbb{R}^{n-1}}|\varphi|^{\frac{2(n-1)}{n-2}}d\sigma\right)^{\frac{n-2}{n-1}}\\
\leq& \frac{1}{Y_{a,b}(\mathbb{R}^n_+,\mathbb{R}^{n-1})}\frac{4(n-1)}{n-2}\int_{\mathbb{R}^n_+}|\nabla \varphi|^2dy,
\end{align*}
equality holds if and only if $\varphi(y)=W(y)$ up to dilations, translations in variables $y^1,\cdots,y^{n-1}$ and any nonzero constant multiple.
\item [(ii)] For all $\e>0$ there exists $\rho_0$ independent of $x_0$ such that for any $\rho \in (0,\rho_0)$ and $\varphi$, which is a smooth function with compact support in a coordinate neighborhood $B_\rho(x_0)\cap\bar M$,
\begin{align*}
&a\left(\int_M |\varphi|^{\frac{2n}{n-2}}d\mu_{g_0}\right)^{\frac{n-2}{n}}+2(n-1)b\left(\int_{\partial M}|\varphi|^{\frac{2(n-1)}{n-2}}d\sigma_{g_0}\right)^{\frac{n-2}{n-1}}\\
\leq &\frac{1+\e}{Y_{a,b}(\mathbb{R}^n_+,\mathbb{R}^{n-1})}\frac{4(n-1)}{n-2}\int_M|\nabla\varphi|_{g_0}^2d\mu_{g_0}.
\end{align*}
\item [(iii)] Given $\e>0$, there exists $C(\e)$ such that for every $\varphi\in H^1(M,g_0)$
\begin{align*}
&a\left(\int_M |\varphi|^{\frac{2n}{n-2}}d\mu_{g_0}\right)^{\frac{n-2}{n}}+2(n-1)b\left(\int_{\partial M}|\varphi|^{\frac{2(n-1)}{n-2}}d\sigma_{g_0}\right)^{\frac{n-2}{n-1}}\\
\leq&\frac{1+\e}{Y_{a,b}(\mathbb{R}^n_+,\mathbb{R}^{n-1})}\frac{4(n-1)}{n-2}\int_M |\nabla \varphi|^2_{g_0}d\mu_{g_0}+C(\e)\int_{M}\varphi^2d\mu_{g_0}.
\end{align*}
\end{enumerate}
\end{lemma}
\begin{proof}
Assertion (i) is a direct consequence of \cite[Theorem 3.3]{escobar5} and \cite[Theorem 1.2]{li-zhu}. 

(ii) Note that $g_0$ is Euclidean in $B_{\rho}(x_0)\cap \bar M$ up to order two under the normal coordinates around $x_0 \in M$ or order one under the Fermi coordinates around $x_0 \in \pa M$. Then the inequality follows from (i) for every $\varphi$ compactly supported in this coordinate chart. 

(iii) This can be proved by a cut-and-paste argument. First choose a finite covering of $\bar M$ by local coordinate charts, each of which satisfies the condition in part (ii). Through an argument of a partition of unity subordinate to this covering, the desired Sobolev inequality follows (e.g.  \cite{aubin_book}).
\end{proof}

\begin{lemma}\label{lem:minimizers_subcritical}
For any $1<q<\frac{n+2}{n-2}$, there exists a smooth positive minimizer $u_q$ for $\mu_q$.
\end{lemma}
\begin{proof}
Let $\{u_i\}\subset H^1(M,g_0)$ be a minimizing sequence of nonnegative functions for $\mu_q$ with the normalization:
$$a \left(\int_Mu_i^{q+1}d\mu_{g_0}\right)^{\frac{2}{q+1}}+2(n-1)b\left(\int_{\partial M}u_i^{\frac{q+3}{2}}d\sigma_{g_0}\right)^{\frac{4}{q+3}}=1, \forall~ i \in \mathbb{N}.$$
It is routine to show $u_i$ is uniformly bounded in $H^1(M,g_0)$. Up to a subsequence, $u_i \rightharpoonup u_q$ in $H^1(M,g_0)$ and $u_i \to u_q$ in $L^{q+1}(M,g_0)$ and $L^{(q+3)/2}(\pa M,g_0)$ as $i \to \infty$. Thus we obtain
\begin{equation}\label{eq:normalization_subcritical}
a \left(\int_Mu_q^{q+1}d\mu_{g_0}\right)^{\frac{2}{q+1}}+2(n-1)b\left(\int_{\partial M}u_{q}^{\frac{q+3}{2}}d\sigma_{g_0}\right)^{\frac{4}{q+3}}=1.
\end{equation}
Then it follows from Lemma \ref{lem:est_H^1_fcns} and \eqref{eq:normalization_subcritical} that
\begin{equation}\label{est:lbd_u_q_M}
\int_M u_q^{q+1}d\mu_{g_0}\geq C_0>0.
\end{equation}
Next we claim that $\int_{\pa M} u_q^{(q+3)/2}d\sigma_{g_0}>0$. By contradiction, if $u_q=0$ a.e. on $\pa M$, namely $u_q \in H_0^1(M,g_0)$, then it yields
$$\mu_q=E[u_q]=\inf_{0\not \equiv v \in H_0^1(M,g_0)}\frac{E[v]}{a \left(\int_M |v|^{q+1}d\mu_{g_0}\right)^{\frac{2}{q+1}}}.$$
Thus the nonnegative minimizer $u_q \in H_0^1(M,g_0)$ weakly solves 
\begin{align*}
\begin{cases}
\displaystyle-\frac{4(n-1)}{n-2}\Delta_{g_0}v+R_{g_0}v=\mu_q a^{\frac{q+1}{2}}v^{q}&\text{ in } M,\\
\displaystyle\frac{\partial v}{\partial{\nu_{g_0}}}=0&\text{ on }\partial M.
\end{cases}
\end{align*}
Hence a contradiction is reached by using Hopf boundary point lemma and \eqref{est:lbd_u_q_M}.

Consequently $u_q$ is a nonzero, nonnegative minimizer with normalization \eqref{eq:normalization_subcritical} for $\mu_q$. Then $u_q \in H^1(M,g_0)$ weakly solves
\begin{align}\label{eq:Euler_Lang_subcritical}
\begin{cases}
\displaystyle-\frac{4(n-1)}{n-2}\Delta_{g_0}u_q+R_{g_0}u_q=\mu_q a \left(\int_Mu_q^{q+1}d\mu_{g_0}\right)^{\frac{1-q}{1+q}}u_q^{q}&\text{ in } M,\\
\displaystyle\frac{2}{n-2}\frac{\partial u_q}{\partial{\nu_{g_0}}}+h_{g_0}u_q=\mu_q b\left(\int_{\partial M}u_q^{q+1}d\sigma_{g_0}\right)^{\frac{1-q}{q+3}}u_q^{\frac{q+1}{2}}&\text{ on }\partial M.
\end{cases}
\end{align}
Then the strong maximum principle gives $u_q>0$ in $\bar M$. Furthermore, a regularity theorem in \cite{cherrier} shows $u_q$ is smooth in $\bar M$.
\end{proof}

\begin{proof}[\textbf{Proof of Theorem \ref{Thm:minimizers}}]
It follows from Lemma \ref{lem:minimizers_subcritical} that for each $1 < q<\frac{n+2}{n-2}$, there exists a positive minimizer $u_q \in H^1(M,g_0)$ with the normalization \eqref{eq:normalization_subcritical}, which solves \eqref{eq:Euler_Lang_subcritical}, namely for all $\psi \in H^1(M,g_0)$,
\begin{align}\label{eq:subcritical_sols}
&\int_M \left(\frac{4(n-1)}{n-2}\langle\nabla u_q ,\nabla \psi\rangle_{g_0}+R_{g_0}u_q\psi\right)d\mu_{g_0}+2(n-1)\int_{\pa M}h_{g_0}u_q\psi d\sigma_{g_0}\no\\
&-\mu_q\left[a \alpha_q\int_M u_q^q \psi d\mu_{g_0}+2(n-1)b\beta_q \int_{\pa M}u_q^{\frac{q+1}{2}}\psi d\sigma_{g_0}\right]=0,
\end{align}
where $\alpha_q=\left(\int_Mu_q^{q+1}d\mu_{g_0}\right)^{(1-q)/(1+q)}, ~\beta_q=\left(\int_{\partial M}u_q^{(q+3)/2}d\sigma_{g_0}\right)^{(1-q)/(q+3)}$.
It follows from Lemma \ref{lem:subcritical_invariants} and \eqref{eq:normalization_subcritical} that $u_q$ is uniformly bounded in $H^1(M,g_0)$. Up to a subsequence, $u_q$ weakly converges to some nonnegative function $u$ in $H^1(M,g_0)$ as $q \nearrow \frac{n+2}{n-2}$. Meanwhile, by Lemma \ref{lem:subcritical_invariants} we get $\mu_q \to Y_{a,b}(M,\pa M)$ as $q \nearrow \frac{n+2}{n-2}$. 

It follows from Lemma \ref{lem:Sobolev_trace_mfld} that for any $\e>0$ there exists $C(\e)>0$ such that
\begin{align*}
&a \left(\int_Mu_q^\frac{2n}{n-2}d\mu_{g_0}\right)^{\frac{n-2}{n}}+2(n-1)b\left(\int_{\partial M}u_q^{\frac{2(n-1)}{n-2}}d\sigma_{g_0}\right)^{\frac{n-2}{n-1}}\\
\leq& (Y_{a,b}(\mathbb{R}_+^n,\mathbb{R}^{n-1})^{-1}+\e)\frac{4(n-1)}{n-2}\int_{M}|\nabla u_q|_{g_0}^2d\mu_{g_0}+C(\e)\int_Mu_q^2d\mu_{g_0}.
\end{align*}
By H\"older's inequality, we have
\begin{align*}
\left(\int_Md\mu_{g_0}\right)^{\frac{n-2}{n}-\frac{2}{q+1}}\left(\int_{M}u_q^{q+1}d\mu_{g_0}\right)^{\frac{2}{q+1}}\leq \left(\int_M u_q^{\frac{2n}{n-2}}d\mu_{g_0}\right)^{\frac{n-2}{n}},\\
\left(\int_{\partial M}d\sigma_{g_0}\right)^{\frac{n-2}{n-1}-\frac{4}{q+3}}\left(\int_{\pa M}u_q^{\frac{q+3}{2}}d\sigma_{g_0}\right)^{\frac{4}{q+3}}\leq \left(\int_{\pa M} u_q^{\frac{2(n-1)}{n-2}}d\sigma_{g_0}\right)^{\frac{n-2}{n-1}}.
\end{align*}
By choosing $q$ sufficiently close to $\frac{n+2}{n-2}$ and using the normalization \eqref{eq:normalization_subcritical}, we get
\begin{align*}
&1-\e\\
\leq& (Y_{a,b}(\mathbb{R}_+^n,\mathbb{R}^{n-1})^{-1}+\e)\frac{4(n-1)}{n-2}\int_{M}|\nabla u_q|_{g_0}^2d\mu_{g_0}+C(\e)\int_Mu_q^2d\mu_{g_0}\\
=& (Y_{a,b}(\mathbb{R}_+^n,\mathbb{R}^{n-1})^{-1}+\e)\left(\mu_q-\int_M R_{g_0}u_q^2d\mu_{g_0}-2(n-1)\int_{\partial M}h_{g_0}u_q^2d\sigma_{g_0}\right)\\
&+C(\e)\int_Mu_q^2d\mu_{g_0}\\
\leq& (Y_{a,b}(\mathbb{R}_+^n,\mathbb{R}^{n-1})^{-1}+2\e)Y_{a,b}(M,\pa M)+C\int_Mu_q^2d\mu_{g_0},
\end{align*}
where the last inequality follows from Lemmas \ref{lem:Sobolev_trace_mfld} and \ref{lem:est_H^1_fcns}. By choosing $\e$ small enough and  the assumption $Y_{a,b}(M,\pa M)<Y_{a,b}(\mathbb{R}_+^n,\mathbb{R}^{n-1})$, we get  
\begin{align*}
\int_Mu_q^2d\mu_{g_0}\geq C_1>0,
\end{align*}
where $C_1$ is independent of $q$. So $\alpha_q$ is uniformly bounded, then after passing to a subsequence we let $\bar \alpha=\lim_{q \nearrow \frac{n+2}{n-2}}\alpha_q>0$. Meanwhile using $u_q\to u$ in $L^2(M,g_0)$ as $q \nearrow \frac{n+2}{n-2}$, we obtain
\begin{equation}\label{est:lbd_L^2}
\int_M u^2d\mu_{g_0}>0.
\end{equation}
 Next we claim that with a constant $C_2$ independent of $q$, there holds 
 $$\int_{\pa M} u_q^{\frac{2(n-1)}{n-2}} d\sigma_{g_0}\geq C_2>0.$$
By negation, there exists a sequence $\{u_q\}$ such that
$$\lim_{q \nearrow \frac{n+2}{n-2}}\int_{\pa M} u_q^{\frac{q+3}{2}}d\sigma_{g_0}=0,$$
then we obtain $\int_{\pa M}u^2 d\sigma_{g_0}=\lim_{q \nearrow \frac{n+2}{n-2}}\int_{\pa M}u_q^2d\sigma_{g_0}=0$, which implies $u=0$ a. e. on $\pa M$. On the other hand, for any $\psi \in H^1(M,g_0)$, we get
$$\beta_q \left|\int_{\pa M} u_q^{\frac{q+1}{2}}\psi d\sigma_{g_0}\right|\leq \left(\int_{\pa M}u_q^{\frac{q+3}{2}}d\sigma_{g_0}\right)^{\frac{2}{q+3}}\|\psi\|_{L^{\frac{q+3}{2}}(M,g_0)}\to 0,$$
as $q \to \infty$. By letting $q \nearrow \frac{n+2}{n-2}$ in equation \eqref{eq:subcritical_sols}, $u$ weakly solves
\begin{align*}
\begin{cases}
\displaystyle-\frac{4(n-1)}{n-2}\Delta_{g_0}u+R_{g_0}u=a\bar \alpha Y_{a,b}(M,\pa M) u^{\frac{n+2}{n-2}}&\text{ in } M,\\
\displaystyle\frac{\partial u}{\partial{\nu_{g_0}}}+\frac{n-2}{2}h_{g_0}u=0&\text{ on }\partial M.
\end{cases}
\end{align*}
Together with \eqref{est:lbd_L^2}, Hopf boundary point lemma gives $u>0$ on $\pa M$. Hence we reach a contradiction.
 
Consequently, after passing to a further subsequence, we let $\lim_{q \nearrow \frac{n+2}{n-2}}\beta_q=\bar \beta>0.$ Furthermore, Fatou's lemma gives
$$\bar \alpha\leq \left(\int_Mu^{\frac{2n}{n-2}}d\mu_{g_0}\right)^{-\frac{2}{n}} ,~~\bar \beta\leq \left(\int_{\partial M}u^{\frac{2(n-1)}{n-2}}d\sigma_{g_0}\right)^{-\frac{1}{n-1}}.$$
Letting $q \nearrow \frac{n+2}{n-2}$ in \eqref{eq:subcritical_sols}, we obtain
\begin{align}\label{eq:critical_sols}
&\int_M \left(\frac{4(n-1)}{n-2}\langle\nabla u,\nabla \psi\rangle_{g_0}+R_{g_0}u\psi\right)d\mu_{g_0}+2(n-1)\int_{\pa M}h_{g_0}u\psi d\sigma_{g_0}\no\\
&-Y_{a,b}(M,\pa M)\left[a \bar \alpha\int_M u^{\frac{n+2}{n-2}} \psi d\mu_{g_0}+2(n-1)b\bar \beta \int_{\pa M}u^{\frac{n}{n-2}}\psi d\sigma_{g_0}\right]=0,
\end{align}
for all $\psi \in H^1(M,g_0)$. The strong maximum principle gives $u>0$ in $\bar M$. Test \eqref{eq:critical_sols} with $u$, it yields
\begin{align*}
&Y_{a,b}(M,\pa M)\\
\leq&\mathcal{Q}_{a,b}[u]=\frac{Y_{a,b}(M,\pa M)\left[a \bar \alpha\int_M u^{\frac{2n}{n-2}}d\mu_{g_0}+2(n-1)b\bar \beta \int_{\pa M}u^{\frac{2(n-1)}{n-2}}d\sigma_{g_0}\right]}{a\left(\int_{M}u^{\frac{2n}{n-2}}d\mu_{g_0}\right)^{\frac{n-2}{n}}+2(n-1)b\left(\int_{\partial M}u^{\frac{2(n-1)}{n-2}}d\sigma_{g_0}\right)^{\frac{n-2}{n-1}}}\\
\leq& Y_{a,b}(M,\pa M).
\end{align*}
From this, we conclude that 
$$\bar \alpha=\left(\int_Mu^{\frac{2n}{n-2}}d\mu_{g_0}\right)^{-\frac{2}{n}} ,~~\bar \beta=\left(\int_{\partial M}u^{\frac{2(n-1)}{n-2}}d\sigma_{g_0}\right)^{-\frac{1}{n-1}}$$
and $Y_{a,b}(M,\pa M)=\mathcal{Q}_{a,b}[u]=E[u]$. Then $u_q \to u$ in $H^1(M,g_0)$ as $q \nearrow \frac{n+2}{n-2}$ and $u$ is a positive minimizer for $Y_{a,b}(M,\pa M)$ and weakly solves \eqref{eq:critical_exponent}. The regularity of $u$ can follow from a theorem by Cherrier \cite{cherrier}.
\end{proof}

\section{Compactness of minimizers for various (a,b)}\label{Sect4}

For brevity, we denote by $u_{a,b}$ the smooth positive minimizer of $Y_{a,b}(M,\pa M)$ with the normalization
\begin{equation}\label{normalization_minimizer}
a \left(\int_M u_{a,b}^{\frac{2n}{n-2}}d\mu_{g_0}\right)^{\frac{n-2}{n}}+2(n-1)b\left(\int_{\pa M}u_{a,b}^{\frac{2(n-1)}{n-2}}d\sigma_{g_0}\right)^{\frac{n-2}{n-1}}=1.
\end{equation}
Under the conformal change of $g=u_{a,b}^{4/(n-2)}g_0$, we have
$$R_g=aY_{a,b}(M,\pa M)\left(\int_M u_{a,b}^\frac{2n}{n-2}d\mu_{g_0}\right)^{-\frac{2}{n}}$$
and
$$ h_g=b Y_{a,b}(M,\pa M)\left(\int_{\pa M}u^{\frac{2(n-1)}{n-2}}_{a,b}d\sigma_{g_0}\right)^{-\frac{1}{n-1}}.$$
Modulo a positive constant multiple, we get $R_g=1$ and 
\begin{equation}\label{eq:normalized_mc}
h_g=\frac{b}{\sqrt{a}}\sqrt{Y_{a,b}(M,\pa M)}\left(\int_M u_{a,b}^\frac{2n}{n-2}d\mu_{g_0}\right)^{\frac{1}{n}}\left(\int_{\pa M}u^{\frac{2(n-1)}{n-2}}_{a,b}d\sigma_{g_0}\right)^{-\frac{1}{n-1}}.
\end{equation}

Let $K$ be a compact set of $\{(a,b); a\geq 0,b\geq 0\}\setminus\{(0,0)\}$. 

\begin{proposition}\label{prop:continuity_conformal_invariant}
Assume $Y(M,\pa M)\geq 0$ and let $(a,b)\in K$,  then $Y_{a,b}(M,\pa M)$ is non-increasing in $a$ for any fixed $b$, as well as in $b$ for any fixed $a$, and is continuous in $K$.
\end{proposition}

\begin{proof} The proof is in the spirit of that of \cite[Proposition 3.2]{escobar3}. For simplicity, we only prove the assertions for $a$ with fixed $b$, the others are similar.
Notice that $Y(M,\pa M) \geq 0$, then $E[u] \geq 0$ for any $u \in H^1(M,g_0)$. For $0\leq a_1\leq a_2$, $Y_{a_1,b}(M,\pa M)\geq Y_{a_2,b}(M,\pa M)$ follows from
\begin{align*}
\mathcal{Q}_{a_1,b}[u]\geq \mathcal{Q}_{a_2,b}[u] \hbox{~~for any~~} u\in H^1(M,g_0).
\end{align*}

Next we prove the continuity of $Y_{a,b}(M,\pa M)$ in $K$. Since $Y(M,\pa M)\geq 0$ we may assume the background metric $g_0$ satisfies $R_{g_0}=0$ in $M$ and $h_{g_0}\geq 0$ on $\pa M$. Suppose $\{(a_m,b_m); m\in \mathbb{N}\}\subset K$ and $(a_m,b_m)\to (a,b) \in K$ as $m\to \infty$. We assume $a \geq 0, b>0$ for simplicity. On one hand, given any $\e>0$, there exists a $u\in H^1(M,g_0)\setminus\{0\}$ such that $\mathcal{Q}_{a,b}[u]<Y_{a,b}(M,\pa M)+\e$. For this fixed $u$, $\mathcal{Q}_{a_m,b_m}[u]\to \mathcal{Q}_{a,b}[u]$ as $m\to \infty$. Then
\begin{equation*}
\lim_{m\to \infty}Y_{a_m,b_m}(M,\pa M)\leq \lim_{m\to \infty}\mathcal{Q}_{a_m,b_m}[u]=\mathcal{Q}_{a,b}[u]<Y_{a,b}(M,\pa M)+\e.
\end{equation*}

On the other hand, given any $\e>0$, for each $(a_m,b_m)$ there exists $u_m \in H^1(M,g_0)$ with
$$a_m\left(\int_Mu_m^{\frac{2n}{n-2}}d\mu_{g_0}\right)^{\frac{n-2}{n}}+2(n-1)b_m\left(\int_{\partial M}u_m^{\frac{2(n-1)}{n-2}}d\sigma_{g_0}\right)^{\frac{n-2}{n-1}}=1$$
 such that $E[u_m]<Y_{a_m,b_m}(M,\pa M)+\e$.  

Let $0\leq a_0=\inf_m a_m$ and $0<b_0=\inf_m b_m$. Then it follows from the monotonicity of $Y_{a,b}(M,\pa M)$ that
$$Y_{a_m,b_m}(M,\pa M)\leq Y_{a_m,b_0}(M,\pa M)\leq  Y_{a_0,b_0}(M,\pa M).$$
From the above normalization of $u_m$, we get 
\begin{align*}
&\frac{4(n-1)}{n-2}\int_M |\nabla u_m|_{g_0}^2d\mu_{g_0}=E[u_m]-2(n-1)\int_{\pa M}h_{g_0}u_m^2 d\sigma_{g_0}\\
\leq& Y_{a_0,b_0}(M,\pa M)+\e+C\int_{\partial M}u^2_md\sigma_{g_0}\leq Y_{a_0,b_0}(M,\pa M)+C.
\end{align*}
This yields $\{u_m\}$ is uniformly bounded in $H^1(M,g_0)$. Thus for all  sufficiently large $m$, we have
$$a\left(\int_Mu_m^{\frac{2n}{n-2}}d\mu_{g_0}\right)^{\frac{n-2}{n}}+2(n-1)b\left(\int_{\partial M}u_m^{\frac{2(n-1)}{n-2}}d\sigma_{g_0}\right)^{\frac{n-2}{n-1}}>1-\e.$$
Consequently, we obtain 
\begin{equation*}
Y_{a,b}(M,\pa M)\leq \mathcal{Q}_{a,b}[u_m]<\frac{E[u_m]}{1-\e}<\frac{Y_{a_m,b_m}(M,\pa M)+\e}{1-\e}
\end{equation*}
for all sufficiently large $m$.
\end{proof}

\begin{lemma}\label{lem:lbd_minimizers}
Suppose $Y_{a,b}(M,\pa M)<Y_{a,b}(\mathbb{R}_+^n,\mathbb{R}^{n-1})$ for all $(a,b)\in K$. Let $u_{a,b}$ be any smooth positive minimizer for $Y_{a,b}(M,\pa M)$ satisfying the normalization \eqref{normalization_minimizer}, then there exists $C=C(K,g_0)>0$ such that 
$$\int_{M}u_{a,b}^{2}d\mu_{g_0}\geq C,\quad \int_{\partial M}u_{a,b}^{\frac{2(n-1)}{n-2}}d\sigma_{g_0}\geq C,\quad \forall (a,b)\in K.$$
\end{lemma}
\begin{proof} For $a=0$, the desired assertions are guaranteed by \cite[Proposition 2.1]{escobar1}. So in the following we assume $a>0$. Given any $\e>0$, Lemma \ref{lem:Sobolev_trace_mfld} gives 
\begin{align*}
&Y_{a,b}(M,\pa M)^{-1}E[u_{a,b}]\\
=&a \left(\int_Mu^\frac{2n}{n-2}d\mu_{g_0}\right)^{\frac{n-2}{n}}+2(n-1)b\left(\int_{\partial M}u^{\frac{2(n-1)}{n-2}}d\sigma_{g_0}\right)^{\frac{n-2}{n-1}}\\
\leq& (Y_{a,b}(\mathbb{R}_+^n,\mathbb{R}^{n-1})^{-1}+\e)\frac{4(n-1)}{n-2}\int_{M}|\nabla u_{a,b}|_{g_0}^2d\mu_{g_0}+C(\varepsilon)\int_Mu^2_{a,b}d\mu_{g_0}.
\end{align*}
Since $Y(M,\pa M)\geq 0$, we can choose an initial metric such that $R_{g_0}\geq 0$ and $h_{g_0}\geq 0$. It follows from Proposition \ref{prop:continuity_conformal_invariant} that $Y_{a,b}(M,\pa M)$ is continuous in $K$, then there exists $k_0>0$ such that
$$\min_K \{Y_{a,b}(\mathbb{R}_+^n,\mathbb{R}^{n-1})-Y_{a,b}(M,\pa M))\}\geq k_0.$$ 
By choosing $\e$ sufficiently small, with a constant $C=C(k_0)>0$ we obtain
\begin{equation}\label{est:inverse_Poincare}
\int_{M}|\nabla u_{a,b}|_{g_0}^2d\mu_{g_0}\leq C \int_M u^2_{a,b}d\mu_{g_0}.
\end{equation}
First we claim that  $\forall (a,b)\in K, \int_M u_{a,b}^2 d\mu_{g_0}\geq \bar C_1(K,g_0)>0$. Otherwise there exists a sequence of minimizers $u_m:=u_{a_m,b_m}$ with $(a_m,b_m)\in K$ such that $\int_M u_m^2d\mu_{g_0}\to 0$, then \eqref{est:inverse_Poincare} gives $\|u_m\|_{H^1(M,g_0)}\to 0$ as $m \to \infty$, which contradicts the normalization \eqref{normalization_minimizer} of $u_m$.

Next we assert that $\forall (a,b)\in K, \int_{\pa M} u_{a,b}^{2(n-1)/(n-2)} d\sigma_{g_0}\geq C(K,g_0)>0$. By negation, there exists a sequence of minimizers $u_m=u_{a_m,b_m}$ with $(a_m,b_m)\in K$ such that $\int_{\pa M}u_{m}^{2(n-1)/(n-2)}d\sigma_{g_0}\to 0$ as $m \to \infty$ . Since $K$ is compact, we may assume $(a_m,b_m)\to (a,b)\in K$ and $E[u_m]=Y_{a_m,b_m}(M,\pa M)\to Y_{a,b}(M,\pa M)$ by Proposition \ref{prop:continuity_conformal_invariant} as $m \to \infty$. Notice that $E[u_m]=Y_{a_m,b_m}(M,\pa M)$, it follows from Proposition \ref{prop:continuity_conformal_invariant} that $u_m$ is uniformly bounded in $H^1(M,g_0)$. Up to a subsequence, there hold $u_m\rightharpoonup u$ weakly in $H^1(M,g_0)$ and 
\begin{align*}
\int_{M}u^2d\mu_{g_0}=& \lim_{m\to \infty}\int_M u_m^2d\mu_{g_0}\geq \bar C_1,\\
\int_{\pa M}u^2d\sigma_{g_0}=&\lim_{m \to \infty}\int_{\pa M}u_m^2d\sigma_{g_0}=0.
\end{align*}
This means $u\not \equiv 0$ and $u=0$ a. e. on $\partial M$. On the other hand, $u_m$ satisfies
\begin{align}\label{eq:weak_sol_Euler_lagrange}
&\int_{M}\left(\frac{4(n-1)}{n-2}\langle\nabla u_m,\nabla\psi\rangle_{g_0}+ R_{g_0}u_m\psi \right)d\mu_{g_0}+2(n-1)\int_{\pa M}h_{g_0}u_m\psi d\sigma_{g_0}\no\\
=&\left[2(n-1)b\left(\int_{\pa M}u_m^{\frac{2(n-1)}{n-2}}d\sigma_{g_0}\right)^{-\frac{1}{n-1}}\int_{\pa M}u_m^{\frac{n}{n-2}}\psi d\sigma_{g_0}\right.\no\\
&\quad\left.+a\left(\int_M u_m^{\frac{2n}{n-2}}d\sigma_{g_0}\right)^{-\frac{2}{n}}\int_{M}u_m^{\frac{n+2}{n-2}}\psi d\sigma_{g_0}\right]Y_{a_m,b_m}(M,\pa M)
\end{align}
for all $\psi \in H^1(M,g_0)$. By H\"older's inequality and the normalization \eqref{normalization_minimizer} for $u_m$, we have
$$
\left(\int_{\pa M}u_m^{\frac{2(n-1)}{n-2}}d\sigma_{g_0}\right)^{-\frac{1}{n-1}}\int_{\pa M}u_m^{\frac{n}{n-2}}\psi d\sigma_{g_0}\to 0
$$
and
$$
\int_M u_m^{\frac{2n}{n-2}}d\mu_{g_0}\to a^{\frac{n}{2-n}},  \hbox{~~as~~} m \to \infty.
$$
By letting $m\to \infty$ in \eqref{eq:weak_sol_Euler_lagrange}, $u$ weakly solves
 \begin{align*}
\begin{cases}
\displaystyle-\frac{4(n-1)}{n-2}\Delta_{g_0}u+R_{g_0}u=a^{\frac{n}{n-2}}Y_{a,b}(M,\pa M) u^{\frac{n+2}{n-2}}&\text{ in } M,\\
\displaystyle\frac{2}{n-2}\frac{\partial u}{\partial{\nu_{g_0}}}+h_{g_0}u=0&\text{ on }\partial M.
\end{cases}
\end{align*}
Then Hopf boundary point lemma gives $u>0$ on $\pa M$. This yields a contradiction.
\end{proof}

Based on these preparations, we are now in a position to establish Theorem \ref{thm:compactness_minimizers}.

\begin{proof}[\textbf{Proof of Theorem \ref{thm:compactness_minimizers}}]
We only need to prove the assertion for $Y(M,\pa M)\geq 0$ due to the same reason of \cite{ChenHoSun}. 

First we claim that there exits $C=C(K,g_0)$ such that $u_{a,b}\leq C$ for any $(a,b)\in K$. By contradiction, suppose there exist sequences $\{(a_m,b_m); m\in \mathbb{N}\} \subset K$ and $\{p_m; m \in \mathbb{N}\} \subset \bar M$ such that
$$r_m:=u_{a_m,b_m}(p_m)=\max_{x \in \bar M} u_{a_m,b_m}(x)\to \infty \hbox{~~as~~} m\to \infty.$$
For brevity, we set $u_m=u_{a_m,b_m}$. Since $M$ is compact, we may assume $p_m\to p_0\in \bar{M}$ as $m \to \infty$. 

If $\lim_{m \to \infty}{\rm dist}_{g_0}(p_m,\pa M)r_m^{\frac{2}{n-2}}=\infty$, under normal coordinates around $p_0$, near $p_0$ there holds
$$(g_0)_{ij}(x)=\delta_{ij}+O(|x|^2).$$
Observe that
\begin{align*}
\frac{4(n-1)}{n-2}\frac{1}{\sqrt{\det g_0}}\pa_i(\sqrt{\det g_0}g^{ij}_0\pa_j u_m)-R_{g_0}u_m+\tilde{a}_mu_m^{\frac{n+2}{n-2}}=0 
\end{align*}
in $\Omega_\rho$, where 
\begin{align*}
\tilde{a}_m=a_m\left(\int_{M}u_m^{\frac{2n}{n-2}}d\mu_{g_0}\right)^{-\frac{2}{n}}Y_{a_m,b_m}(M,\pa M).
\end{align*}
Define $\rho_m=\rho r_m^{\frac{2}{n-2}}$ and
$$v_m(y)=r_m^{-1}u_m(\exp_{p_m}(yr_m^{-\frac{2}{n-2}})) \text{ for }y\in B_{\rho_m}(0)\subset \mathbb{R}^n.$$
Then $v_m(0)=1$ and $0<v_m(y)\leq 1$ in $B_{\rho_m}(0)$. Let $g_m(y)=g_0(\exp_{p_m}(yr_m^{-\frac{2}{n-2}}))$, $f_m(y)=r_m^{-\frac{4}{n-2}}R_{g_0}(\exp_{p_m}(yr_m^{-\frac{2}{n-2}}))$. Then $v_m$ satisfies
\begin{align*}
\frac{4(n-1)}{n-2}\frac{1}{\sqrt{\det g_m}}\frac{\partial }{\partial y^i}(\sqrt{\det g_m}g^{ij}_m\frac{\partial}{\partial y^j} v_m)-f_mv_m+\tilde{a}_mv_m^{\frac{n+2}{n-2}}=0
\end{align*}
in $B_{\rho_m}(0)$. As $m\to \infty$, there hold
$$(g_m)_{ij}\to\delta_{ij}\quad f_m\to 0 \text{ in }C^1(\hat K) \text{~~for any compact set~~} \hat K \subset\mathbb{R}^n.$$
Since $K$ is compact and from Lemma \ref{lem:lbd_minimizers} that $\tilde{a}_m$ is uniformly bounded, up to a subsequence we get
$$(a_m,b_m)\to (a,b),\quad \tilde{a}_m\to \tilde a, \hbox{~~as~~} m \to \infty.$$
From the $W^{2,p}$-estimate, $\|v_m\|_{C^\lambda(B_{r_m})}$ is uniformly bounded  for any  $\lambda\in (0,1)$. Applying Schauder interior estimates and the diagonal method to extract a subsequence from $\{v_m\}$,  still denote as $\{v_m\}$, we obtain $v_m\to v$ in $C^{2,\lambda}(\hat K)$, as $m \to \infty$. Moreover $v$ satisfies
\begin{align*}
\frac{4(n-1)}{n-2}\Delta v+\tilde{a}v^{\frac{n-2}{n+2}}=0 \hbox{~~in~~} \mathbb{R}^n.
\end{align*}
Notice that $v(0)=1$ and $0\leq v\leq 1$, the strong maximum principle gives $v>0$. From Fatou's lemma, we have
\begin{equation}\label{eq:vol_upperbd_v}
\int_{\mathbb{R}^n}v^{\frac{2n}{n-2}}dx\leq\liminf_{m\to \infty}\int_{B_{\rho_m}(0)}v_m^{\frac{2n}{n-2}} \sqrt{\det g_m}dx\leq\liminf_{m\to \infty} \int_{M}u_m^{\frac{2n}{n-2}}d\mu_{g_0}.
\end{equation}
Recall that 
\begin{align*}
\tilde{a}=aY_{a,b}(M,\pa M)\lim_{m\to \infty}\left(\int_{M}u_m^{\frac{2n}{n-2}}d\mu_{g_0}\right)^{-\frac{2}{n}}.
\end{align*}
It is not hard to show that if $Y(M,\pa M)=0$, then $Y_{a,b}(M,\pa M)=0$ for any $(a,b)\in K$. If either $a=0$ or $Y(M,\pa M)=0$, then $\tilde a=0$. Then the strong maximum principle gives $v\equiv 1$. Using similar arguments in Lemma \ref{lem:lbd_minimizers}, one can show that $u_m$ is uniformly bounded in $H^1(M,g_0)$. From this and \eqref{eq:vol_upperbd_v}, we have
$$\int_{\mathbb{R}^n}v^{\frac{2n}{n-2}}dx\leq C,$$
which contradicts $v\equiv 1$ in $\mathbb{R}^n$. When $Y(M,\pa M)>0$ and $a>0$, then $\tilde a>0$. Hence the classification theorem in \cite{Caffarelli-Gidas-Spruck} states that any $C^2$ positive solution $v$ to the above problem has the form of
$$v_{x_0,\epsilon}(x)=\left(\frac{4n(n-1)}{\tilde a}\right)^{\frac{n-2}{4}}\left(\frac{\epsilon}{\epsilon^2+|x-x_0|^2}\right)^{\frac{n-2}{2}}$$
for some $\epsilon>0$ and fixed $x_0 \in \mathbb{R}^n$. It is well-know that
$$ \int_{\mathbb{R}^n}v_{x_0,\epsilon}^{\frac{2n}{n-2}}dx= \int_{\mathbb{R}^n}v_{0,1}^{\frac{2n}{n-2}}dx, \quad  \int_{\mathbb{R}^n}|\nabla v_{x_0,\epsilon}|^2dx=\int_{\mathbb{R}^n}|\nabla v_{0,1}|^2dx.$$
Then we assert that
\begin{align}
\tilde{a}\int_{\mathbb{R}^n}v^{\frac{2n}{n-2}}dx=&\frac{4(n-1)}{n-2}\int_{\mathbb{R}^n}|\nabla v|^2dx=2^{\frac{2}{n}}aY_{a,0}(\mathbb{R}^n_+,\mathbb{R}^{n-1})\left(\int_{\mathbb{R}^n}v^{\frac{2n}{n-2}}dx\right)^{\frac{n-2}{n}}.\label{eq:contr-arg-1}
\end{align}
where the last identity follows from
\begin{align*}
Y_{a,0}(\mathbb{R}_+^n,\mathbb{R}^{n-1})=&\frac{\frac{4(n-1)}{n-2}\int_{\mathbb{R}_+^n}|\nabla v_{0,1}|^2 dx}{a\left(\int_{\mathbb{R}_+^n}v_{0,1}^{\frac{2n}{n-2}}dx\right)^{\frac{n-2}{n}}}=\frac{\frac{4(n-1)}{n-2}\frac{1}{2}\int_{\mathbb{R}^n}|\nabla v_{0,1}|^2 dx}{a\left(\frac{1}{2}\int_{\mathbb{R}^n}v_{0,1}^{\frac{2n}{n-2}}dx\right)^{\frac{n-2}{n}}}\\
=&\frac{1}{2^{\frac{2}{n}}a}\frac{\frac{4(n-1)}{n-2}\int_{\mathbb{R}^n}|\nabla v|^2 dx}{\left(\int_{\mathbb{R}^n}v^{\frac{2n}{n-2}}dx\right)^{\frac{n-2}{n}}},
\end{align*}
in view of the classification theorem in \cite{li-zhu} and symmetry.
Together with Proposition \ref{prop:continuity_conformal_invariant}, \eqref{eq:vol_upperbd_v} and \eqref{eq:contr-arg-1} give
\begin{align*}
Y_{a,0}(\mathbb{R}^n_+,\mathbb{R}^{n-1})\geq Y_{a,0}(M,\pa M)\geq& Y_{a,b}(M,\pa M)\geq 2^{\frac{2}{n}}Y_{a,0},(\mathbb{R}^n_+,\mathbb{R}^{n-1}),
\end{align*}
which obviously yields a contradiction.

If $\lim_{m \to \infty}{\rm dist}_{g_0}(p_m,\pa M)r_m^{\frac{2}{n-2}}<\infty$. Let $X=(x^1,\cdots,x^{n-1})$ be the normal coordinates of $x \in \pa M$ around $p_0$ and $\nu(X):=\nu_{g_0}(X)$ be the unit outward normal at $x \in \pa M$. For small $t\geq 0$, $\exp_{X}(-t\nu(X)):B_\rho^+(0)\to \Omega_\rho\subset M$ is a diffeomorphism, then  $(x^1,\cdots,x^{n-1},t)$ are called the Fermi coordinates around $p_0$. Without loss of generality, we assume $p_m\in \Omega_\rho$ and denote by $p_m=\exp_{X_m}(-t_m\nu(X_m))$.

Under these coordinates, we have
\begin{align*}
\begin{cases}
\displaystyle\frac{4(n-1)}{n-2}\frac{1}{\sqrt{\det g_0}}\pa_i(\sqrt{\det g_0}g^{ij}_0\pa_j u_m)-R_{g_0}u_m+\tilde{a}_mu_m^{\frac{n+2}{n-2}}=0 &\hbox{~~in~~}\Omega_\rho,\\
\displaystyle\frac{2}{n-2}\frac{\partial u_m}{\partial \nu_{g_0}}+h_{g_0}u_m=\tilde{b}_mu_m^{\frac{n}{n-2}}&\hbox{~~on~~}\pa \Omega_\rho\cap \pa M,
\end{cases}
\end{align*}
where 
\begin{align*}
\tilde{a}_m&=a_m\left(\int_{M}u_m^{\frac{2n}{n-2}}d\mu_{g_0}\right)^{-\frac{2}{n}}Y_{a_m,b_m}(M,\pa M),\\
\tilde{b}_m&=b_m\left(\int_{\partial M}u_m^{\frac{2(n-1)}{n-2}}d\sigma_{g_0}\right)^{-\frac{1}{n-1}}Y_{a_m,b_m}(M,\pa M).
\end{align*}
Define $\rho_m=\rho r_m^{\frac{2}{n-2}}$ and
$$v_m(X,t)=r_m^{-1}u_m(\exp_{X_m}(-t_mr_m^{-\frac{2}{n-2}} \nu(Xr_m^{-\frac{2}{n-2}}))) \hbox{~~in~~}B_{\rho_m}^+(0).$$
Then $v_m(0)=1$ and $0<v_m(X,t)\leq 1$ in $B_{\rho_m}^+(0)$. We set
\begin{align*}
g_m(X,t)=&g_0\big(\exp_{X_m}(-t_mr_m^{-\frac{2}{n-2}}\nu(Xr_m^{-\frac{2}{n-2}}))\big),\\
\tilde f_m(X,t)=&r_m^{-\frac{4}{n-2}}R_{g_0}\big(\exp_{X_m}(-t_mr_m^{-\frac{2}{n-2}}\nu(Xr_m^{-\frac{2}{n-2}}))\big),\\h_m(X)=&r_m^{-\frac{2}{n-2}}h_{g_0}(Xr_m^{-\frac{2}{n-2}}).
\end{align*}
 Thus $v_m$ satisfies
\begin{align*}
\begin{cases}
\displaystyle \frac{4(n-1)}{n-2}\frac{1}{\sqrt{\det g_m}}\partial_i(\sqrt{\det g_m}g^{ij}_m\partial_j v_m)-\tilde f_mv_m+\tilde{a}_mv_m^{\frac{n+2}{n-2}}=0 &\hbox{in~~}B_{\rho_m}^+,\\
\displaystyle -\frac{2}{n-2}\pa_t v_m+h_mv_m-\tilde{b}_mv_m^{\frac{n}{n-2}}=0 &\hbox{on~~} D_{\rho_m}.
\end{cases}
\end{align*}
Since $r_m\to \infty$ as $m \to \infty$, we have  
$$(g_m)_{ij}\to\delta_{ij},\quad \tilde f_m,~~h_m\to 0 \text{~~in~~}C^1(\tilde K)$$
for any compact set $\tilde K \subset \overline{\mathbb{R}_+^n}$. Since $K$ is compact and from Lemma \ref{lem:lbd_minimizers} that $\tilde{a}_m,\tilde b_m$ are bounded, up to a subsequence we have
$$(a_m,b_m)\to (a,b),~~\tilde{a}_m\to \tilde a,~~\tilde{b}_m\to \tilde{b} \hbox{~~as~~} m \to \infty.$$
From $W^{2,p}$-estimate, $\|v_m\|_{C^\lambda(\overline{B_{\rho_m}^+})}$ is uniformly bounded  for any  $\lambda\in (0,1)$.  Applying the Schauder estimates and the diagonal method to extract a subsequence from $\{v_m\}$,  still denote as $\{v_m\}$, we obtain $v_m\to v$ in $C^{2,\lambda}(\tilde K)$, as $m \to \infty$. Moreover $v$ satisfies
\begin{align}\label{eq:blowupsols_wbdry}
\begin{cases}
\displaystyle\frac{4(n-1)}{n-2}\Delta v+\tilde{a}v^{\frac{n-2}{n+2}}=0&\text{~~in~~}\mathbb{R}^n_+,\\
\displaystyle-\frac{\partial v}{\partial t}-\tilde{b}v^{\frac{n}{n-2}}=0&\text{~~on~~}\mathbb{R}^{n-1}.
\end{cases}
\end{align}
Notice that $v(0)=1$ and $0\leq v\leq 1$, the strong maximum principle gives $v>0$. By Proposition \ref{prop:continuity_conformal_invariant} we get
\begin{align*}
\tilde{a}=&aY_{a,b}(M,\pa M)\lim_{m\to \infty}\left(\int_{M}u_m^{\frac{2n}{n-2}}d\mu_{g_0}\right)^{-\frac{2}{n}},\\
\tilde{b}=&bY_{a,b}(M,\pa M)\lim_{m\to \infty}\left(\int_{\partial M}u_m^{\frac{2(n-1)}{n-2}}d\sigma_{g_0}\right)^{-\frac{1}{n-1}}.
\end{align*}
Fatou's  lemma gives
\begin{align*}
\int_{\mathbb{R}^n_+}v^{\frac{2n}{n-2}}dx\leq\liminf_{m\to \infty}\int_{B_{\rho_m}^+}v_m^{\frac{2n}{n-2}}\sqrt{\det g_m}dx\leq\liminf_{m\to \infty} \int_{M}u_m^{\frac{2n}{n-2}}d\mu_{g_0},\\
\int_{\mathbb{R}^{n-1}}v^{\frac{2(n-1)}{n-2}}d\sigma\leq\liminf_{m\to \infty}\int_{D_{\rho_m}}v_m^{\frac{2(n-1)}{n-2}}\sqrt{\det g_m}d\sigma\leq\liminf_{m\to \infty} \int_{\partial M}u_m^{\frac{2(n-1)}{n-2}}d\sigma_{g_0}.
\end{align*}
If $Y(M,\pa M)=0$, then $\tilde a=\tilde b=0$. Then the strong maximum principle gives $v\equiv 1$ in $\mathbb{R}_+^n$. As above, we also get $v \in L^{2n/(n-2)}(\mathbb{R}_+^n)$. Thus we reach a contradiction. If $Y(M,\pa M)>0$, testing with $v$ in problem \eqref{eq:blowupsols_wbdry}, we get
\begin{align*}
&\tilde{a}\int_{\mathbb{R}^n_+}v^{\frac{2n}{n-2}}dx+2(n-1)\tilde{b}\int_{\mathbb{R}^{n-1}}v^{\frac{2(n-1)}{n-2}}d\sigma\\
=&\frac{4(n-1)}{n-2}\int_{\mathbb{R}^n_+}|\nabla v|^2dx\\
=&\left[a\left(\int_{\mathbb{R}^n_+}v^{\frac{2n}{n-2}}dx\right)^{\frac{n-2}{n}}+2(n-1)b\left(\int_{\mathbb{R}^{n-1}}v^{\frac{2(n-1)}{n-2}}d\sigma\right)^{\frac{n-2}{n-1}}\right]Y_{a,b}(\mathbb{R}_+^n,\mathbb{R}^{n-1}),
\end{align*}
where the last identity follows from the fact that $Y_{a,b}(\mathbb{R}^+,\mathbb{R}^{n-1})$ is achieved by any positive solution to \eqref{eq:blowupsols_wbdry} in virtue of \cite[Theorem 3.3]{escobar5}.
Since $a^2+b^2>0$ and $Y(M,\pa M)>0$ imply $\tilde a^2+\tilde b^2>0$, combining with the above estimates we have
$$Y_{a,b}(M,\pa M)\geq Y_{a,b}(\mathbb{R}^n_+,\mathbb{R}^{n-1}),$$
which contradicts the assumption $Y_{a,b}(M,\pa M)<Y_{a,b}(\mathbb{R}_+^n,\mathbb{R}^{n-1}),~\forall~ (a,b) \in K$.

Finally based on the above upper bound, it follows from Lemma \ref{lem:lbd_minimizers} and \cite[Proposition A-4]{almaraz5} that $\forall ~(a,b)\in K$, $u_{a,b}$ has a uniform positive  lower bound. 
Then the Schauder estimates give the $C^2$-estimate of $u_{a,b}$ in $K$.
\end{proof}

As a byproduct of Proposition \ref{prop:continuity_conformal_invariant}, there hold
\begin{align*}
\lim_{a\to 0^+} Y_{a,b}(M,\pa M)&=Y_{0,b}(M,\pa M),\quad \text{for any fixed } b>0,\\
\lim_{b\to 0^+} Y_{a,b}(M,\pa M)&=Y_{a,0}(M,\pa M),\quad \text{for any fixed } a>0.
\end{align*}
From these together with Theorem \ref{thm:compactness_minimizers}, when $Y(M, \pa M)>0$ expression \eqref{eq:normalized_mc} shows that the normalized conformal metric of scalar curvature $1$ has positive constant mean curvature, which runs in a large set of $\mathbb{R}_+$.

\section{Construction of test functions}\label{Sect5}
In this section, we use the following notation: given any $\rho>0$, let
\begin{align*}
B_\rho^+(0)=B_\rho(0)\cap \mathbb{R}^n_+;\qquad &\partial^+ B_\rho^+(0)=\partial B_\rho^+(0)\cap \mathbb{R}^n_+;\\
D_\rho(0)=\partial B_\rho^+(0)\backslash \partial^+ B_\rho^+(0).&
\end{align*}

From now on, we assume $Y(M,\pa M)>0$. Recall that $d=[(n-2)/2]$ when $n \geq 3$. By a result of Marques \cite{marques2}, for each $x_0\in \partial M$ there exists a conformal metric $g_{x_0}=f_{x_0}^{4/(n-2)}g_0$ with $f_{x_0}(x_0)=1$. Suppose  $\Psi_{x_0}:B_{2\rho}^+(0)\to M$ is the $g_{x_0}$-Fermi coordinates around $x_0$, set $x=\Psi_{x_0}(y)$ for $y \in B_{2\rho}^+(0)$.  Under these coordinates, there hold  $\det g_{x_0}=1+O(|y|^{2d+2})$, $(g_{x_0})_{ij}(0)=\delta_{ij}$ and $(g_{x_0})_{ni}(y)=\delta_{ni}$, for any $y\in B^+_{2\rho}(0)$ and $i,j=1,...,n$. Let $g_{x_0}=\exp(h)$, where $\exp$ denotes the matrix exponential, then the symmetric 2-tensor $h$ has the following properties:
\begin{equation}\label{propr:h}
\begin{cases}
\mathrm{tr}\,h(y)=O(|y|^{2d+2})\,,&\text{for}\: y\in B_{2\rho}^+(0),\\
h_{ab}(0)=0\,,&\text{for}\: i,j=1,...,n\,,
\\
h_{in}(y)=0\,,&\text{for}\:y\in B^+_{2\rho}(0),\: i=1,...,n\,,
\\
\d_ah_{bc}(0)=0\,,&\text{for}\: a,b,c=1,...,n-1\,,
\\
\sum_{b=1}^{n-1}y^bh_{ab}(y)=0\,,&\text{for}\:y\in D_{2\rho}(0),\: a=1,...,n-1\,.
\end{cases}
\end{equation}  
The last two properties follow from the fact that Fermi coordinates are normal on $\pa M$. 

\begin{convention} 
In the following, we let $a,b,c,\cdots$ range from $1$ to $n-1$ and $i,j,k\cdots$ range from $1$ to $n$. We adopt Einstein summation convention and simplify $B_\rho^+(0),\partial^+ B_\rho^+(0)$, $D_\rho(0)$ by $B_\rho^+$, $\partial^+ B_\rho^+$, $D_\rho$ without otherwise stated.
\end{convention}

Under these conformal Fermi coordinates, the mean curvature satisfies
\begin{align}\label{mean_curv_Fermi}
h_{g_{x_0}}(x)=&-\frac{1}{2(n-1)}g^{ab}\d_ng_{ab}(x)\no\\
=&-\frac{1}{2(n-1)}\d_n(\log \det (g_{x_0}))(x)=O(|y|^{2d+1}).
\end{align}
Let $H_{ij}$ be the Taylor expansion of $h_{ij}$ up to order $d$, namely
$$H_{ij}=\sum_{|\alpha|=1}^d \pa^\alpha h_{ij}y^\alpha,$$
where $\alpha$ is a multi-index and $\pa^\alpha h_{ij}=\pa^\alpha h_{ij}(0)$. Then $H$ satisfies \eqref{propr:h} except the first property replaced by ${\rm tr}H=0$.

\subsection{Linearization of scalar curvature and mean curvature}\label{Subsect5.1}
By \eqref{eq:bdry_bubble} and \eqref{prob:half-space} we get
\begin{equation}\label{eq:bubble_Einstein}
W_\e\pa_i\pa_j W_\e-\frac{n}{n-2}\pa_iW_\e\pa_jW_\e=\frac{1}{n}\left(W_\e\Delta W_\e-\frac{n}{n-2}|\nabla W_\e|^2\right)\delta_{ij} \hbox{~~in~~} \mathbb{R}_+^n.
\end{equation}

\begin{proposition}\label{prop:linearized_eqs_smc}
Let $V$ be a smooth vector field in $\overline{\mathbb{R}_+^n}$ satisfying $V_n=0=\pa_n V_a $ on $ \mathbb{R}^{n-1}$, where $1\leq a \leq n-1$. Let 
$$\psi=V_k \pa_k W_\e+\frac{n-2}{2n}W_\e \mathrm{div} V $$
and
$$S_{ij}=\pa_i V_j+\pa_j V_i-\frac{2}{n} \mathrm{div} V \delta_{ij}$$
be a conformal Killing operator. Then we have
\begin{equation}\label{eq:linearized_scalar_curv}
\Delta \psi+n(n+2)W_\e^{\frac{4}{n-2}}\psi=\frac{n-2}{4(n-1)}W_\e \pa_i\pa_j S_{ij}+\pa_i(\pa_j W_\e S_{ij})\hbox{~~in~~} \mathbb{R}_+^n
\end{equation}
and
\begin{equation}\label{eq:linearized_mean_curv}
\partial_n \psi-\frac{n}{n-2}W_\epsilon^{-1}\partial_n W_\epsilon \psi=\frac{1}{2}\partial_{n}W_\epsilon S_{nn}+\frac{n-2}{4(n-1)}W_\epsilon\partial_n S_{nn} \hbox{~~on~~} \mathbb{R}^{n-1}.
\end{equation}
\end{proposition}
\begin{proof}
The linearized equations \eqref{eq:linearized_scalar_curv} and \eqref{eq:linearized_mean_curv} for scalar curvature and mean curvature can be verified by direct computations in \cite[Proposition 5]{Brendle2} and \cite[Proposition 5]{ChenSophie}, respectively. Somewhat inspired by Brendle \cite{Brendle2}, we adopt a geometric proof of these linearized equations. It involves the first variation formulae for scalar curvature and mean curvature at a round metric of the spherical cap $\Sigma$.

Let $g_\Sigma=W_\e^{4/(n-2)}g_{\mathbb{R}^n}$ be the standard spherical metric on $\Sigma$ of constant sectional curvature $4$, see also Section \ref{Sect2}. We now consider a family of perturbed metrics of $g_\Sigma$:
\begin{equation}\label{eq:perturbed_metrics}
W_\e^{\frac{4}{n-2}}e^{tS}=\phi_t^\ast((W_\e-t\psi)^{\frac{4}{n-2}}g_{\mathbb{R}^n}), ~~t \in \mathbb{R},
\end{equation}
where $\phi_t$ is one-parameter family of diffeomorphisms on $S^n$ generated by $V$. Differentiating \eqref{eq:perturbed_metrics} with respect to $t$ and evaluating at $t=0$, we get
\begin{equation}\label{eq:decom_S}
W_\epsilon^{\frac{4}{n-2}} S=\mathcal{L}_{V}(g_\Sigma)-\frac{4}{n-2}\psi W_\e^{-1}g_\Sigma.
\end{equation}
We remark that such a decomposition of symmetric $2$-tensor is guaranteed by \cite[Lemma 4.57]{Besse}. Recall that the first variation of scalar curvature (see \cite[Theorem 1.174 (e)]{Besse}) is given by:
\begin{equation}\label{eq:1st_var_R}
R_g'(h)=-h^{ik}R_{ik}+\nabla^i\nabla^k h_{ik}-\Delta_g {\rm tr}_g (h)
\end{equation}
for any symmetric $2$-tensor $h$, where $\nabla$ indicates the covariant derivative of $g$.

On one hand, letting $\tilde g_E=e^{tS}$ we have
$$R_{W_\epsilon^{\frac{4}{n-2}}\tilde g_E}=W_\epsilon^{-\frac{n+2}{n-2}}\left(-\frac{4(n-1)}{n-2}\Delta_{\tilde g_E}W_\epsilon+R_{\tilde g_E}W_\epsilon\right).$$
 Notice that $\det \tilde g_E=1$ due to ${\rm tr}S=0$, then
\begin{align*}
\frac{d}{dt}\Big|_{t=0}\Delta_{\tilde g_E}W_\epsilon=\frac{d}{dt}\Big|_{t=0}\partial_i(e^{-tS_{ij}}\partial_j W_\epsilon)=-\partial_i(S_{ij}\partial_jW_\epsilon)
\end{align*}
and \eqref{eq:1st_var_R} gives
$$\frac{d}{dt}\Big|_{t=0}R_{\tilde g_E}=\partial_i\partial_j S_{ij}.$$
Thus we obtain
\begin{align}\label{eq:w_first_way}
R_{g_\Sigma}'(W_\epsilon^{\frac{4}{n-2}} S)=&\frac{d}{dt}\Big|_{t=0}R_{W_\epsilon^{\frac{4}{n-2}}\tilde g_E}\notag\\
=&W_\epsilon^{-\frac{n+2}{n-2}}\left(\frac{4(n-1)}{n-2}\partial_i(S_{ij}\partial_jW_\epsilon)+\partial_i\partial_jS_{ij}W_\epsilon\right).
\end{align}
On the other hand, using \eqref{eq:decom_S} and \eqref{eq:1st_var_R}, we have
\begin{equation}\label{relation:DR}
R_{g_\Sigma}'(W_\epsilon^{\frac{4}{n-2}}S)=R_{g_\Sigma}'(\mathcal L_{V}(g_\Sigma))-R_{g_\Sigma}'(\frac{4}{n-2}\psi W_\epsilon^{-1}g_\Sigma),
\end{equation}
where $\mathcal L_{V}(g_\Sigma)$ denotes the Lie derivative of metric $g_\Sigma$ along the vector field $V$. In particular, it is routine to verify that
\begin{equation}\label{est:R'_Lie_derivative}
R_{g_\Sigma}'(\mathcal L_{V}(g_\Sigma))=0.
\end{equation}
\iffalse
\textcolor{blue}{To see this, let $g=g_\Sigma$ and $h_{ij}=\mathcal{L}_V(g)=V_{i,j}+V_{j,i}$,  then $R_{ij}=4(n-1)g_{ij}$. Under $g$-normal coordinates, Ricci identity gives\todo[inline]{Here $V_i=g_{ij}V^j=W_\e^{\frac{4}{n-2}}V_i$.} 
$$V_{i,ji}=V_{i,ij}+R_{iji}^kV_k=V_{i,ij}+R_{kj}V_k=V_{i,ij}+4(n-1)V_j.$$
It follows that
$$V_{i,jij}=V_{i,ijj}+4(n-1)V_{j,j}$$
and
\begin{align*}
V_{j,iij}=&V_{j,iji}+R_{jij}^k V_{k,i}+R_{iij}^kV_{j,k}=V_{j,iji}+R_{ki}V_{k,i}-R_{kj}V_{j,k}\\
=&V_{j,iji}=(V_{j,ji}+4(n-1)V_i)_{,i}=V_{j,jii}+4(n-1)V_{i,i}.
\end{align*}
Combining these we obtain
\begin{align*}
(R_{g_\Sigma})'(h)=&-h_{ij}R_{ij}+h_{ij,ij}-h_{ii,jj}\no\\
=&-8(n-1) \mathrm{div} V+(V_{i,j}+V_{j,i})_{,ij}-2V_{i,ijj}\\
=&-8(n-1) \mathrm{div} V+(V_{i,jij}+V_{j,iij})-2V_{i,ijj}\\
=&0.
\end{align*}
This proves the claim.
}
\fi
It also follows from \eqref{eq:1st_var_R} that
\begin{align}\label{eq:w_second_way}
&R_{g_\Sigma}'(\frac{4}{n-2}\psi W_\epsilon^{-1}g_\Sigma)\no\\
=&\frac{4}{n-2}\left[-4n(n-1)W_\epsilon^{-1}\psi+(1-n)\Delta_{g_\Sigma}(W_\epsilon^{-1}\psi)\right]\no\\
=&-\frac{4(n-1)}{n-2}\left[4nW_\epsilon^{-1}\psi+n(n-2) W_\epsilon^{-1} \psi+W_\epsilon^{-\frac{n+2}{n-2}} \Delta \psi\right]\no\\
=&-\frac{4(n-1)}{n-2}W_\epsilon^{-\frac{n+2}{n-2}}\left[n(n+2)W_\epsilon^{\frac{4}{n-2}}\psi+\Delta \psi\right].
\end{align}
\iffalse
\textcolor{blue}{Another way to see this, let $g=g_\Sigma$ and $h_{ij}=\frac{4}{n-2}\bar u_\epsilon^{-1}w_\epsilon g_{S}$. Under $g$-normal coordinates, we have
\begin{align*}
h_{ij,ij}=&\frac{4}{n-2}\Delta_{g_\Sigma}(\bar u_\epsilon^{-1}w_\epsilon)\\
h_{ii,jj}=&\frac{4n}{n-2}\Delta_{g_\Sigma}(\bar u_\epsilon^{-1}w_\epsilon)
\end{align*}
Then we obtain
\begin{align*}
(R_{g_\Sigma})'(\frac{4}{n-2}w_\epsilon \bar u_\epsilon^{-1}g_\Sigma)=&\frac{4}{n-2}\left[-4n(n-1)\bar u_\epsilon^{-1}w_\epsilon+(1-n)\Delta_{g_\Sigma}(\bar u_\epsilon^{-1}w_\epsilon)\right]\\
=&-\frac{4(n-1)}{n-2}[4n\bar u_\epsilon^{-1}w_\epsilon+n(n-2) \bar u_\epsilon^{-1} w_\epsilon+\bar u_\epsilon^{-\frac{n+2}{n-2}} \Delta w_\epsilon]\\
=&-\frac{4(n-1)}{n-2}\bar u_\epsilon^{-\frac{n+2}{n-2}}[n(n+2)\bar u_\epsilon^{\frac{4}{n-2}}w_\epsilon+\Delta w_\epsilon].
\end{align*}
}
\fi
Putting \eqref{eq:w_first_way}-\eqref{eq:w_second_way} together, we obtain equation \eqref{eq:linearized_scalar_curv}.

Next we need to show \eqref{eq:linearized_mean_curv}. Let $\nu_g$ be the outward unit normal on $\mathbb{R}^{n-1}$, then
$$\nu_g=-\frac{g^{ni}}{\sqrt{g^{nn}}}\partial_i$$
and
\begin{align}\label{eq:mean_curv_local}
h_g=&-\frac{1}{n-1}g^{ab}\langle\nu_g,\nabla_{\pa_a}\pa_b\rangle=\frac{1}{n-1}g^{ab}g^{ni}(g^{nn})^{-\frac{1}{2}}g_{ij}\Gamma_{ab}^j\no\\
=&\frac{1}{n-1}g^{ab}\Gamma^n_{ab}(g^{nn})^{-\frac{1}{2}}.
\end{align}

\iffalse
\textcolor{blue}{ Notice that $\langle\nu_g,\pa_a\rangle=0$ and $\langle \nu_g,\nu_g\rangle=\frac{1}{g^{nn}}\langle g^{ni}\pa_{x^i},g^{nj}\pa_{x^j}\rangle=\frac{1}{g^{nn}}g^{ni}g_{ij}g^{nj}=1$. There is another way to compute:
\begin{align*}
h_g=&\frac{1}{n-1}g^{ab}\langle\nabla_{\partial_a}\nu_g,\partial_b\rangle\no\\
=&-\frac{1}{n-1}\left[\partial_a(g^{ni}(g^{nn})^{-\frac 12})g_{ib}g^{ab}+g^{ni}(g^{nn})^{-\frac 12}\Gamma_{ai}^j g^{ab}g_{jb}\right]\no\\
=&-\frac{1}{n-1}\left[\partial_a(g^{na}(g^{nn})^{-\frac 12})+g^{ni}(g^{nn})^{-\frac 12}\Gamma_{ai}^a\right].
\end{align*}
}
\fi

By the conformal change formula of mean curvatures, we get
\begin{equation}\label{eq:var_mean_curv_g_S}
\frac{d}{dt}\Big|_{t=0}h_{W_\epsilon^{\frac{4}{n-2}}\tilde g_E}=\frac{2}{n-2}W_\epsilon^{-\frac{n}{n-2}}\frac{d}{dt}\Big|_{t=0}\left(\frac{\partial W_\epsilon}{\partial\nu_{\tilde g_E}}+\frac{n-2}{2}h_{\tilde g_E}W_\epsilon\right).
\end{equation}
Observe that
$$\frac{\partial W_\epsilon}{\partial \nu_{\tilde g_E}}=-(\tilde g^{nn}_E)^{-1/2}\tilde g^{ni}_E\partial_i W_\epsilon,$$
then
\begin{equation}\label{eq:var_normal_derivative}
\frac{d}{dt}\Big|_{t=0}\frac{\partial W_\epsilon}{\partial \nu_{\tilde g_E}}=S_{ni}\partial_iW_\epsilon-\frac{1}{2}S_{nn} \partial_nW_\epsilon=\frac{1}{2}S_{nn}\pa_nW_\e,
\end{equation}
where the last identity follows from $S_{an}=0$ on $\mathbb{R}^{n-1}$ due to the assumption that $V_n=0=\pa_nV_a$ on $\mathbb{R}^{n-1}$. 
Recall that the Christoffel symbols of $\tilde g_E$ are given by
\begin{align*}
\tilde\Gamma_{ab}^n=\frac{1}{2}\tilde g_E^{ni}\left[\pa_b (\tilde g_E)_{ai}+\pa_a(\tilde g_E)_{ib}-\pa_i(\tilde g_E)_{ab}\right]
\end{align*}
then
$$\frac{d}{dt}\Big|_{t=0}\tilde \Gamma_{ab}^n=-\frac{1}{2}\pa_n S_{ab},$$
due to $S_{an}=0$ on $\mathbb{R}^{n-1}$.
\iffalse
\textcolor{blue}{Since the Christoffel symbols of metric $\tilde g_E$ are defined by 
\begin{align*}
\tilde\Gamma_{ai}^a=\frac{1}{2}\tilde g_E^{ab}\left[\pa_i (\tilde g_E)_{ba}+\pa_a(\tilde g_E)_{bi}-\pa_b(\tilde g_E)_{ai}\right],
\end{align*}
then
$$\frac{d}{dt}\Big|_{t=0}\tilde \Gamma_{ai}^a=\frac{1}{2}\pa_i S_{aa}=-\frac{1}{2}\pa_i S_{nn},$$
where the last identity follows from $-S_{nn}=S_{aa}$ due to ${\rm tr S}=0$.}
\fi
From this and \eqref{eq:mean_curv_local}, we get
\begin{equation}\label{eq:var_mean_curv_near_Euclidean}
\frac{d}{dt}\Big|_{t=0}h_{\tilde g_E}=-\frac{1}{2(n-1)}\pa_n S_{aa}=\frac{1}{2(n-1)}\partial_n S_{nn},
\end{equation}
where the last identity follows from $-S_{nn}=S_{aa}$ due to $\mathrm{tr S}=0$.
Plugging \eqref{eq:var_normal_derivative} and \eqref{eq:var_mean_curv_near_Euclidean} into \eqref{eq:var_mean_curv_g_S}, we obtain
\begin{equation}\label{eq:var_h_1st}
\frac{d}{dt}\Big|_{t=0}h_{W_\epsilon^{\frac{4}{n-2}}\tilde g_E}=\frac{2}{n-2}W_\epsilon^{-\frac{n}{n-2}}\left(\frac{1}{2}\partial_{n}W_\epsilon S_{nn}+\frac{n-2}{4(n-1)}W_\epsilon\partial_n S_{nn}\right).
\end{equation}

On the other hand, using \eqref{eq:decom_S} we have
\begin{equation}\label{relation:Dh}
h_{g_\Sigma}'(W_\epsilon^{\frac{4}{n-2}}S)=h_{g_\Sigma}'(\mathcal{L}_V(g_\Sigma))-h_{g_\Sigma}'(\frac{4}{n-2}\psi W_\epsilon^{-1}g_\Sigma) \hbox{~~on~~} \mathbb{R}^{n-1}.
\end{equation}
First we assert that
\begin{equation}\label{est:h'_Lie_derivative}
h_{g_\Sigma}'(\mathcal{L}_V(g_\Sigma))=0 \qquad \hbox{~~on~~} \mathbb{R}^{n-1}.
\end{equation}
Next we compute
\begin{align}\label{eq:var_h_2nd}
&\frac{d}{dt}\Big|_{t=0}h_{(W_\epsilon-t\psi)^{\frac{4}{n-2}}g_E}=-\frac{2}{n-2}\frac{d}{dt}\Big|_{t=0}\left[(W_\epsilon-t\psi)^{-\frac{n}{n-2}}\partial_n(W_\epsilon-t\psi)\right]\no\\
=&\frac{2}{n-2}{W_\epsilon^{-\frac{n}{n-2}}}\left(\partial_n \psi-\frac{n}{n-2}W_\epsilon^{-1}\partial_n W_\epsilon \psi\right).
\end{align}
Therefore from \eqref{eq:var_h_1st}-\eqref{eq:var_h_2nd}, equation \eqref{eq:linearized_mean_curv} follows. 

It remains to show assertion \eqref{est:h'_Lie_derivative}. Define
\begin{align*}
\hat S_{ij}:=\mathcal{L}_V(g_\Sigma)_{ij}=(V_k\pa_k W_\e^{\frac{4}{n-2}})\delta_{ij}+W_\e^{\frac{4}{n-2}}(\pa_i V_j+\pa_j V_i).
\end{align*}
For brevity, we abuse $g=g_\Sigma$ for a while. Since $V_n=0=\pa_nV_a$ on $\mathbb{R}_+^n$, then $\hat S_{an}=0$ on $\mathbb{R}_+^n$.  Observe that
$$(\Gamma_{ab}^n)'=\frac{1}{2}g^{ni}(\nabla_b\hat S_{ia}+\nabla_a\hat S_{ib}-\nabla_i\hat S_{ab})=\frac{1}{2}W_\e^{-\frac{4}{n-2}}(\nabla_b\hat S_{na}+\nabla_a\hat S_{nb}-\nabla_n\hat S_{ab}),$$
then
\begin{align*}
g^{ab}(\Gamma_{ab}^n)'=&W_\e^{-\frac{4}{n-2}}\left[g^{ab}\hat S_{na,b}-\frac{1}{2}\pa_n {\rm tr}_g(\hat S)\right]\\
=&W_\e^{-\frac{4}{n-2}}\left[W_\e^{-\frac{4}{n-2}}\hat S_{na,a}-\frac{1}{2}\pa_n(W_\e^{-\frac{4}{n-2}}\hat S_{aa})\right].
\end{align*}
 We compute
\begin{align*}
\hat S_{na,a}=&\pa_a \hat S_{na}-\Gamma_{na}^i \hat S_{ia}-\Gamma_{aa}^i \hat S_{ni}\\
=&-\Gamma_{na}^b \hat S_{ba}-\Gamma_{aa}^n \hat S_{nn}\\
=&-2T_cW_\e^{\frac{2}{n-2}}[\hat S_{aa}-(n-1)\hat S_{nn}],
\end{align*}
where the last identity follows from
\begin{align*}
\Gamma_{na}^b=&\frac{1}{2}g^{bc}\pa_n g_{ca}=\frac{1}{2}W_\e^{-\frac{4}{n-2}}\pa_n W_\e^{\frac{4}{n-2}} \delta_{ab}=2T_cW_\e^{\frac{2}{n-2}}\delta_{ab},\\
\Gamma_{aa}^n=&-\frac{1}{2}g^{nn}\pa_n g_{aa}=-\frac{1}{2}W_\e^{-\frac{4}{n-2}}\pa_n W_\e^{\frac{4}{n-2}} \delta_{aa}=-2(n-1)T_cW_\e^{\frac{2}{n-2}}
\end{align*}
in virtue of \eqref{prob:half-space}. It follows from \eqref{eq:mean_curv_local} that
\begin{align*}
(n-1)(h_g)'(\hat S)=-\hat S^{ab}\pi_{ab}+\frac{n-1}{2}\frac{\hat S^{nn}}{g^{nn}}h_g+\frac{(\Gamma_{ab}^n)'}{\sqrt{g^{nn}}}g^{ab}.
\end{align*}
By \eqref{eq:bubble_Einstein} we get
\begin{align*}
&\pa_n\pa_a W_\e^{\frac{4}{n-2}}=\frac{4}{n-2}\left[\frac{6-n}{n-2}W_\e^{\frac{4}{n-2}-2}\pa_n W_\e \pa_a W_\e+\pa_n \pa_a W_\e\right]\\
=&\frac{4}{n-2}\frac{6}{n-2}W_\e^{\frac{4}{n-2}-2}\pa_n W_\e \pa_a W_\e=6T_c W_\e^{\frac{2}{n-2}}\pa_a W_\e^{\frac{4}{n-2}},
\end{align*}
whence
\begin{align*}
&\pa_n \hat S_{aa}
=\pa_n\left[(n-1)(V_k\pa_k W_\e^{\frac{4}{n-2}})+2W_\e^{\frac{4}{n-2}}\pa_a V_a\right]\\
=&(n-1)(V_a\pa_n\pa_a W_\e^{\frac{4}{n-2}}+\pa_n V_n \pa_n W_\e^{\frac{4}{n-2}})+2\pa_n W_\e^{\frac{4}{n-2}}\pa_a V_a\\
=&(n-1)T_cW_\e^{\frac{2}{n-2}}(6V_a \pa_a W_\e^{\frac{4}{n-2}}+4\pa_n V_n W_\e^{\frac{4}{n-2}})+8T_cW_\e^{\frac{6}{n-2}}\pa_a V_a.
\end{align*}
Then we have
\begin{align*}
&\pa_n(W_\e^{-\frac{4}{n-2}}\hat S_{aa})=W_\e^{-\frac{4}{n-2}}\pa_n \hat S_{aa}+\pa_n W_\e^{-\frac{4}{n-2}} \hat S_{aa}\\
=&(n-1)T_cW_\e^{-\frac{2}{n-2}}(6V_a \pa_a W_\e^{\frac{4}{n-2}}+4\pa_n V_n W_\e^{\frac{4}{n-2}})\\
&+8T_cW_\e^{\frac{2}{n-2}}\pa_a V_a-4T_c W_\e^{-\frac{2}{n-2}}\hat S_{aa}.
\end{align*}
Consequently, we obtain
\begin{align*}
&-T_c^{-1}W_\e^{\frac{6}{n-2}}g^{ab}(\Gamma_{ab}^n)'\\
=&2[\hat S_{aa}-(n-1)\hat S_{nn}]+(n-1)(3V_a \pa_a W_\e^{\frac{4}{n-2}}+2\pa_n V_n W_\e^{\frac{4}{n-2}})\\
&+4W_\e^{\frac{4}{n-2}}\pa_a V_a-2\hat S_{aa}\\
=&-2(n-1)\hat S_{nn}+4W_\e^{\frac{4}{n-2}}\pa_a V_a+(n-1)(3V_a \pa_a W_\e^{\frac{4}{n-2}}+2\pa_n V_n W_\e^{\frac{4}{n-2}}).
\end{align*}
Putting these facts together and using $\pi_{ab}=-2T_c g_{ab}$, we conclude that
\begin{align*}
&(n-1)T_c^{-1}W_\e^{\frac{4}{n-2}}(h_g)'(\hat S)\\
=&2\hat S_{aa}-(n-1)\hat S_{nn}+T_c^{-1}W_\e^{\frac{6}{n-2}}g^{ab}(\Gamma_{ab}^n)'\\
=&2\hat S_{aa}+(n-1)\hat S_{nn}-(n-1)(3V_a \pa_a W_\e^{\frac{4}{n-2}}+2\pa_n V_n W_\e^{\frac{4}{n-2}})-4W_\e^{\frac{4}{n-2}}\pa_a V_a\\
=&2\Big[(n-1)(V_a\pa_a W_\e^{\frac{4}{n-2}})+2W_\e^{\frac{4}{n-2}}\pa_a V_a\Big]\\
&+(n-1)\Big[(V_a\pa_a W_\e^{\frac{4}{n-2}})+2W_\e^{\frac{4}{n-2}}\pa_n V_n\Big]\\
&-(n-1)(3V_a \pa_a W_\e^{\frac{4}{n-2}}+2\pa_n V_n W_\e^{\frac{4}{n-2}})-4W_\e^{\frac{4}{n-2}}\pa_a V_a\\
=&0,
\end{align*}
which implies the desired assertion.
\end{proof}

\subsection{Test functions and their energy estimates}\label{Subsect5.2}

Let $\chi(y)=\chi(|y|)$ be a smooth cut-off function in $\overline{\mathbb{R}_+^n}$ with $\chi=1$ in $B_1^+$ and $\chi=0$ in $\mathbb{R}^n_+\backslash \overline{B_2^+}$. For any $\rho>0$, set $\chi_\rho(y)=\chi(|y|/\rho)$ for $y \in \mathbb{R}_+^n$. As in \cite{Brendle-Chen} and \cite{ChenSophie}, given $H_{ij}$
there exists a smooth vector field $V$ in $\overline{\mathbb{R}_+^n}$ such that
\begin{align}\label{eq:V}
\begin{cases}
\sum\limits_{i=1}^{n}\pa_i\left[\U^{\crit}\Big(\chi_{\rho}H_{ij}-\d_iV_j-\d_jV_i+\frac{2}{n}(\mathrm{div} V)\delta_{ij}\Big)\right]=0,&\mathrm{~~in~~}\:\Rn,
\\
\d_nV_a=V_n=0,&\mathrm{~~on~~}\:\mathbb{R}^{n-1},
\end{cases}
\end{align}
where $1\leq i,j \leq n, 1\leq a\leq n-1$. Moreover, there holds
\begin{align}\label{est:V}
|\d^{\b}V(y)|\leq C(n,T_c, |\b|)\sum_{a,b=1}^{n-1}\sum_{|\a|=1}^{d}|\pa^\alpha h_{ab}|(\e+|y|)^{|\a|+1-|\b|}.
\end{align}

We only sketch the proof of the construction of vector field $V$. Consider the spherical cap $(\Sigma,g_\Sigma)$ as in Proposition \ref{prop:linearized_eqs_smc} with $\e=1$. Define
\begin{align*}
\mathscr{X}=\{V \in H^1(\Sigma,g_\Sigma);  \langle V,\nu_{g_\Sigma}\rangle_{g_\Sigma}=0 \hbox{ for a vector field } V \text{ on }\partial \Sigma\}
\end{align*}
and $\mathscr{H}$ the space of all trace-free symmetric two-tensors on $\Sigma$ of class $L^2$. A conformal Killing operator $\mathcal{D}_{g_\Sigma}: \mathscr{X} \to \mathscr{H}$ on $\Sigma$ is defined as
$$\mathcal{D}_{g_\Sigma}V=\mathcal{L}_V(g_\Sigma)-\frac{2}{n}(\mathrm{div}_{g_\Sigma}V)g_\Sigma.$$
Similarly as in the appendix of \cite{Brendle-Chen}, we know that $\rm{ker}\mathcal{D}_{g_\Sigma}$  is finite dimensional. We define
$$\mathscr{X}_0=\{V\in \mathscr{X};\langle V,Z\rangle_{L^2(\Sigma,g_\Sigma)}=0, \forall\, Z\in \mathrm{ker}\mathcal{D}_{g_\Sigma}\}.$$
Using a similar argument in \cite[Proposition A.3]{Brendle-Chen}, we assert that for any symmetric two-tensor $\tilde h$ with compact support in $\mathbb{R}_+^n$ , there exists a unique vector field $V\in \mathscr{X}_0$ such that 
$$\langle W^{\frac{4}{n-2}}\tilde h-\mathcal{D}_{g_\Sigma}V,\mathcal{D}_{g_\Sigma}Z\rangle_{L^2(\Sigma,{g_\Sigma})}=0 \hbox{~~for all~~} Z\in \mathscr{X}.$$
Furthermore, with a dimensional constant $C$ there holds 
\begin{align*}
\|V\|_{L^2(\Sigma,g_\Sigma)}^2+\|\nabla V\|_{L^2(\Sigma,g_\Sigma)}^2\leq C\|W^{\frac{4}{n-2}}\tilde h\|_{L^2(\Sigma,g_\Sigma)}^2.
\end{align*}
Based on this estimate and using our $W$ instead, we can construct the vector field $V$ satisfying \eqref{eq:V} and estimate \eqref{est:V} by mimicking the proofs of \cite[Propositions 12-13]{ChenSophie}.

As in Proposition \ref{prop:linearized_eqs_smc}, we define symmetric trace-free 2-tensors $S$ and $T$ in $\overline \Rn$ by 
\begin{equation}\label{def:S:T}
S_{ij}=\d_iV_j+\d_jV_i-\frac{2}{n}\mathrm{div}V\delta_{ij}\quad\quad\text{and}\quad\quad T=H-S\,.
\end{equation}
It follows from (\ref{eq:V}) that $T$ satisfies
\begin{equation}\label{eq:U:T:1}
\U\d_jT_{ij}+\frac{2n}{n-2}\d_j\U T_{ij}=0\,,
\:\:\:\:\text{in}\:\:B^+_{2\rho}.
\end{equation}

For $n \geq 3$, we define an auxiliary function $\w=\w_{\e,\rho,{H}}$ 
by
\begin{equation}\label{eq:def:phi}
\w=\d_i\U V_i+\frac{n-2}{2n}\U \mathrm{div} V.
\end{equation}
When $n=3$, then $d=0$ and we choose $\psi=0$.
Using \eqref{est:V} and \eqref{eq:bdry_bubble} of $\U$, in $B_{2\rho}^+$ we have
\begin{align}\label{est:w}
|\w(y)|\leq C(n,T_c)\epsilon^{\frac{n-2}{2}}\sum_{a,b=1}^{n-1}\sum_{|\alpha|=1}^d |\pa^\alpha h_{ab}|(\epsilon+|y|)^{|\alpha|+2-n}.
\end{align}

By the above construction of $V$ and $H_{in}=0$ in $B_{2\rho}^+$, we know that
$$W_\epsilon\partial_i S_{ni}+\frac{2n}{n-2}\partial_i W_\epsilon S_{ni}=0 \hbox{~~in~~}B_{2\rho}^+$$
and $S_{na}=0$ on $D_{2\rho}$.
Thus we get
$$\partial_n S_{nn}=-\partial_a S_{na}-\frac{2n}{n-2}W_\epsilon^{-1}\partial_i W_\epsilon S_{ni}=-\frac{2n}{n-2}W_\epsilon^{-1}\partial_nW_\epsilon S_{nn} \mathrm{~~on~~} D_{2\rho}.$$
 Combining this and \eqref{eq:linearized_mean_curv},  we conclude that
$$\partial_n \psi-\frac{n}{n-2} W_\e^{-1}\partial_n W_\e \psi=-\frac{1}{2(n-1)}\partial_{n} W_\e S_{nn} \quad\mathrm{~~on~~}D_{2\rho}.$$
For future citation, we collect the linearized equations for scalar curvature and mean curvature in the following

\begin{lemma} \label{lem:phi}
The function $\w$ satisfies
\begin{align*}
\begin{cases}
\displaystyle \Delta \w+n(n+2)\U^{\frac{4}{n-2}}\w=\frac{n-2}{4(n-1)}\U \d_i\d_j S_{ij}+\d_i(\d_j\U S_{ij})&\mathrm{in~~}\:B_{2\rho}^+,\\
\displaystyle\d_n \w-\frac{n}{n-2}\U^{-1}\pa_n\U\w=-\frac{1}{2(n-1)}\pa_n\U S_{nn}&\mathrm{on~~}\:D_{2\rho}.
\end{cases}
\end{align*}
\end{lemma}

Similar to \cite[Proposition 5]{ChenSophie}, we collect and derive some properties associated to $S$ and $T$.
\begin{lemma}\label{lem:S_T}
\begin{enumerate}
\item [(1)] On~~$D_{2\rho}~$ there holds $S_{an}=0=T_{an}$~~for~~$0\leq a\leq n-1$.
\item [(2)] On $D_{2\rho}~$ there hold
\begin{align*}
\pa_n S_{nn}=&-\frac{2n}{n-2}\U^{-1}\pa_n\U S_{nn},\\
\pa_n S_{ab}=&-\frac{1}{n-1}\pa_nS_{nn}\delta_{ab},
\end{align*}
where $1\leq a,b\leq n-1$.
\end{enumerate}
\end{lemma}

Based on Lemma \ref{lem:phi}, we rearrange \cite[Propositions 5-6]{Brendle2} as follows.

\begin{proposition}\label{prop:second_variation_identity} There holds
\begin{align*}
&\frac{1}{4}Q_{ik,j}Q_{ik,j}-\frac{1}{2}Q_{ki,k}Q_{li,l}+2\U^{\frac{2n}{n-2}}T_{ik}T_{ik}\\
=&\frac{1}{4}\U^2\partial_lH_{ik}\partial_lH_{ik}-\frac{2(n-1)}{n-2}  \partial_k\U\partial_l\U H_{ik}H_{il}-2\U\partial_k\U H_{ik}\partial_lH_{il}\\
&-\frac{1}{2}\U^2\partial_kH_{ik}\partial_lH_{il}+\frac{8(n-1)}{n-2}\partial_i\U\partial_k\w H_{ik}-\frac{4(n-1)}{n-2}|\nabla\w|^2\\
&+\frac{4(n-1)}{n-2}n(n+2)\U^{\frac{4}{n-2}}\w^2-2\U\w\pa_i\pa_k H_{ik}+\mathrm{div} \xi,
\end{align*}
where 
\begin{align*}
Q_{ij,k}=\U \partial_kT_{ij}+\frac{2}{n-2}(\partial_l\U T_{il}\delta_{jk}+\partial_l\U T_{jl}\delta_{ik}-\partial_i\U T_{jk}-\partial_j\U T_{ik})
\end{align*}
and the vector field $\xi$ is  given by
\begin{align}\label{eq:xi}
\xi_i=&2 \U \w \pa_k H_{ik}-2\U\pa_k \w H_{ik}-2\pa_k \U\w H_{ik}-\frac{1}{2}\U^2\pa_i S_{lk}H_{lk}\no\\
&+\U^2\pa_l S_{kl}H_{ik}+2\U \pa_l \U S_{kl} H_{ik}-\U \w \pa_k S_{ik}+\U \pa_k \w S_{ik}\no\\
&+\pa_k \U \w S_{ik}+\frac{1}{4}\U^2 \pa_i S_{lk}S_{lk}-\frac{1}{2}\U^2 \pa_l S_{kl}S_{ik}-\U \pa_l \U S_{kl} S_{ik}\no\\
&-\frac{4(n-1)}{n-2}\pa_k \U \w S_{ik}+\frac{4(n-1)}{n-2}\w \pa_i \w-\frac{2}{n-2}\U \pa_k \U T_{lk}T_{il}.
\end{align}
In particular, it yields
\begin{align}\label{rel:bdr_deri_xi}
\xi_n=-\frac{n+2}{2(n-2)}\U \pa_n\U S_{nn}^2+\frac{4n(n-1)}{(n-2)^2}\U^{-1}\pa_n\U\w^2
\end{align}
on $\mathbb{R}^{n-1}$.
\end{proposition}

\begin{proposition}\label{prop:Q_and_h}
There exists $\lambda^*=\lambda^*(n,T_c)>0$ such that 
\begin{align*}
&\lambda^*\epsilon^{n-2}\sum_{i,j=1}^n\sum_{|\alpha|=1}^d |\pa^\alpha h_{ij}|^2\int_{B_\rho^+(0)}(\epsilon+|y|)^{2|\alpha|+2-2n}dy\\
\leq & \frac{1}{4}\int_{B_\rho^+(0)}Q_{ij,k}Q_{ij,k}dy-\frac{n^2}{2(n-1)(n-2)}\int_{D_\rho(0)}\partial_n \U\U S_{nn}^2 d\sigma
\end{align*}
for all $\rho\geq 2\epsilon$.
\end{proposition}
\begin{proof}
Since only the unchanged sign condition of $\pa_n\U$ on $D_\rho$ and Lemma \ref{lem:S_T} were required in \cite[Lemma 3.4]{almaraz5}, we refer to similar arguments in \cite[Proposition 3.5]{almaraz5} for the details.
\end{proof}

Our test function is 
\begin{equation}\label{eq:test_fcn}
\ubar=[\chi_\rho(\U+\w)]\circ\Psi_{x_0}^{-1}+(1-\chi_\rho)\circ\Psi_{x_0}^{-1}\epsilon^{\frac{n-2}{2}}G,
\end{equation}
where $G=G_{x_0}$ is the Green's function of the conformal Laplacian with pole at $x_0\in \partial M$, coupled with a boundary condition, namely
\begin{align}\label{def:Green_funct}
\begin{cases}
\displaystyle -\frac{4(n-1)}{n-2}\Delta_{g_{x_0}}G_{x_0}+R_{g_{x_0}}G_{x_0}=0\,,&\text{in}\:M\backslash \{x_0\}\,,
\\
\displaystyle \frac{2}{n-2}\frac{\partial G_{x_0}}{\partial \nu_{g_{x_0}}} +h_{g_{x_0}}G_{x_0}=0\,,&\text{on}\:\d M\backslash \{x_0\}\,.
\end{cases}
\end{align}
We assume that $G$ is normalized such that $\lim_{y\to 0}G(\Psi_{x_0}(y))|y|^{n-2}=1$. Then $G$ satisfies the following estimates near $x_0$, namely for sufficiently small $|y|$ (see \cite[Proposition B-2]{Almaraz-Sun}):
\begin{align}\label{est:Green_funct}
&|G(\Psi_{x_0}(y))-|y|^{2-n}|\no\\
\leq& C\sum_{a,b=1}^{n-1}\sum_{|\alpha|=1}^d |\pa^\alpha h_{ab}||y|^{|\alpha|+2-n}+
\begin{cases}
C|y|^{d+3-n}, \quad~~\text{~~if~~}\:n\geq 5,\\
C(1+|\log|y||),\text{~~if~~}\:n=3,4,
\end{cases}\no\\
&|\nabla (G(\Psi_{x_0}(y))-|y|^{2-n})|\leq C\sum_{a,b=1}^{n-1}\sum_{|\alpha|=1}^d |\pa^\alpha h_{ab}||y|^{|\alpha|+1-n}+C|y|^{d+2-n}.
\end{align}
 Moreover, there holds
\begin{align*}
C(T_c,n)^{-1}\epsilon^{\frac{n-2}{2}}(\epsilon+|y|)^{2-n}\leq \U(y)\leq C(T_c,n)\epsilon^{\frac{n-2}{2}}(\epsilon+|y|)^{2-n}.
\end{align*}

We consider the flux integral as in \cite[P.1006]{Brendle-Chen}
\begin{align*}
\mathcal{I}(x_0,\rho)=&-\int_{\pa^+B_\rho^+}|y|^{2-2n}(|y|^{2}\pa_j h_{ij}-2n y^j h_{ij})\frac{y^i}{|y|}d\sigma\\
&+\frac{4(n-1)}{n-2}\int_{\partial^+ B_\rho^+}(|y|^{2-n}\partial_iG-G\partial_i|y|^{2-n})\frac{y^i}{|y|}d\sigma
\end{align*}
for $x_0\in \partial M$ and all sufficiently small $\rho>0$.

The following estimates on the expansion of scalar curvature could be found in \cite[P. 2645]{almaraz5},  which follows from  \cite[Proposition 11]{Brendle2} and \cite[Proposition 3]{ChenSophie}. Just keep in mind that the boundary is not necessarily umbilic here.
\begin{proposition}\label{prop:expan_scalar}
The scalar curvature $R_{g_{x_0}}$ satisfies
\begin{align*}
&|R_{g_{x_0}}-\pa_i\pa_kH_{ik}|\leq C\sum_{a,b=1}^{n-1}\sum_{ |\alpha|=1}^d |\pa^\alpha h_{ab}||y|^{|\alpha|-1}+C|y|^{d-1},\\
&\left|R_{g_{x_0}}-\pa_i\pa_k h_{ik}+\pa_k(H_{ik}\pa_lH_{il})-\frac{1}{2}\pa_k H_{ik} \pa_l H_{il}+\frac{1}{4}\pa_lH_{ik}\pa_lH_{ik}\right|\\
\leq &C\sum_{a,b=1}^{n-1}\sum_{|\alpha|=1}^d |\pa^\alpha h_{ab}|^2|y|^{2|\alpha|-1}+C\sum_{a,b=1}^{n-1}\sum_{|\alpha|=1}^d |\pa^\alpha h_{ab}||y|^{|\alpha|+d-1}+C|y|^{2d}
\end{align*}
for $|y|$ sufficiently small.
\end{proposition}

In order to prove this theorem, we need to estimate the energy $E[\ubar]$. Notice that
\begin{align*}
E[\ubar]=&\int_M \left(\frac{4(n-1)}{n-2}|\nabla\ubar|_{g_{x_0}}^2+R_{g_{x_0}}\ubar^2\right)d\mu_{g_{x_0}}\\
&+2(n-1)\int_{\pa M}h_{g_{x_0}}\ubar^2 d\sigma_{g_{x_0}}.
\end{align*}
We will estimate $E[\ubar]$ in $\Psi_{x_0}(B_\rho^+)$ and $M\backslash \Psi_{x_0}(B_\rho^+)$ respectively. 
\begin{proposition}\label{prop:near_pole}  With some sufficiently small $\rho_0>0$, there holds
\begin{align*}
&\int_{B^+_\rho}\left[\frac{4(n-1)}{n-2}|\nabla(\U+\w)|_{g_{x_0}}^2+R_{g_{x_0}}(\U+\w)^2\right]dy\\
&+2(n-1)\int_{D_\rho}h_{g_{x_0}}(\U+\w)^2d\sigma\\
\leq &4n(n-1)\int_{B^+_\rho}\U^{\frac{4}{n-2}}\left(\U^2+\frac{n+2}{n-2}\w^2\right)dy\\
&+\int_{\partial^+B_\rho^+}\frac{4(n-1)}{n-2}\partial_i\U\U\frac{y^i}{|y|}d\sigma+\int_{\pa^+B_\rho^+}(\U^2\pa_j h_{ij}-\pa_j\U^2 h_{ij})\frac{y^i}{|y|}d\sigma\\
&-4(n-1)T_c\int_{D_\rho}\U^{\frac{2}{n-2}}\Big(\U^2+2\U\w+\frac{n}{n-2}\w^2-\frac{n-2}{8(n-1)^2}\U^2S_{nn}^2\Big)d\sigma\\
&-\frac{1}{2}\lambda^* \sum_{a,b=1}^{n-1}\sum_{|\alpha|=1}^d |\pa^\alpha h_{ab}|^2\epsilon^{n-2}\int_{B_\rho^+}(\epsilon+|y|)^{2|\alpha|+2-2n}dy\\
&+C\sum_{a,b=1}^{n-1}\sum_{|\alpha|=1}^d |\pa^\alpha h_{ab}|\epsilon^{n-2} \rho^{|\alpha|+2-n}+C\epsilon^{n-2} \rho^{2d+4-n}
\end{align*}
for $0<2\epsilon<\rho<\rho_0\leq 1$, where $\rho_0$  and $C$ are some constants depending only on $n,T_c,g_0$.
\end{proposition}
\begin{proof}
Notice that $\ubar=\U+\w$ in $B_\rho^+$. First it follows from \eqref{mean_curv_Fermi} and \eqref{eq:def:phi} that
\begin{equation}\label{eq:mean_curv_hot}
\int_{D_\rho}h_{g_{x_0}}(\U+\w)^2 d\sigma\leq C \int_{D_\rho}|y|^{2d+1}(\U+\w)^2 d\sigma\leq C \epsilon^{n-2}\rho^{2d+4-n}.
\end{equation}
Next we decompose 
\begin{align}\label{eq:decom_energy_integrand}
&\frac{4(n-1)}{n-2}|\nabla(\U+\w)|^2_{g_{x_0}}+R_{g_{x_0}}(\U+\w)^2\no\\
=&\frac{4(n-1)}{n-2}|\nabla \U|^2+\frac{4(n-1)}{n-2}n(n+2)\U^{\frac{4}{n-2}}\w^2+\sum_{i=1}^4 J_i,
\end{align}
where
\begin{align*}
J_1=&\frac{8(n-1)}{n-2}\pa_i\U\pa_i\w-\frac{4(n-1)}{n-2}\pa_i\U\pa_k\U h_{ik}+\U^2\pa_i\pa_k h_{ik},\\
&\quad-\U^2\pa_k(H_{ik}\pa_lH_{il})-2\U\pa_k\U H_{ik}\pa_lH_{il},\\
J_2=&-\frac{1}{4}\U^2\partial_lH_{ik}\partial_lH_{ik}+\frac{2(n-1)}{n-2}  \partial_k\U\partial_l\U H_{ik}H_{il}+2\U\partial_k\U H_{ik}\partial_lH_{il}\\
&\quad+\frac{1}{2}\U^2\partial_kH_{ik}\partial_lH_{il}+2\U\w\pa_i\pa_k H_{ik}-\frac{8(n-1)}{n-2}\pa_i\U\pa_k\w H_{ik}\\
&\quad+\frac{4(n-1)}{n-2}|\nabla \w|^2-\frac{4(n-1)}{n-2}n(n+2)\U^{\frac{4}{n-2}}\w^2,\\
J_3=&\frac{4(n-1)}{n-2}(g^{ik}_{x_0}-\delta_{ik}+h_{ik}-\frac{1}{2}H_{il}H_{kl})\pa_i\U\pa_k\U\\
&\quad+\Big[R_{g_{x_0}}-\pa_i\pa_k h_{ik}+\pa_k(H_{ik}\pa_lH_{il})-\frac{1}{2}\partial_kH_{ik}\partial_lH_{il}+\frac{1}{4}\pa_lH_{ik}\pa_lH_{ik}\Big]\U^2,\\
J_4=&\frac{8(n-1)}{n-2}(g^{ik}_{x_0}-\delta_{ik}+H_{ik})\pa_i\U\pa_k\w+2(R_{g_{x_0}}-\pa_i\pa_kH_{ik})\U\w\\
&\quad +R_{g_{x_0}}\w^2+\frac{4(n-1)}{n-2}(g^{ik}_{x_0}-\delta_{ik})\pa_i\w\pa_k\w.
\end{align*}

We start with $J_1$. Rearrange $J_1$ as
\begin{align*}
J_1=&\frac{8(n-1)}{n-2}\pa_i(\pa_i\U\w)-\frac{8(n-1)}{n-2}\w\Delta\U+\pa_i(\U^2\pa_kh_{ik})-\pa_k(\pa_i\U^2h_{ik})\\
&+2\left(\U\pa_i\pa_k\U-\frac{n}{n-2}\pa_i\U\pa_k\U\right)h_{ik}-\pa_k(\U^2H_{ik}\pa_lH_{il}).
\end{align*}
Notice that $\U$ satisfies
\begin{align*}
\w\Delta \U =&-\frac{(n-2)^2}{2} \pa_i(\U^{\frac{2n}{n-2}}V_i).
\end{align*}
Thus using $V_n=0$ on $D_\rho$, $H_{in}=h_{in}=0, {\rm tr}\,h=O(|y|^{2d+2})$ in $B_\rho^+$ and \eqref{eq:bubble_Einstein}, we have
\begin{align*}
&\int_{B_\rho^+}J_1dy\\
=&-\frac{8(n-1)}{n-2}\int_{D_\rho}\pa_n\U \w d\sigma+\frac{8(n-1)}{n-2}\int_{\pa^+B_\rho^+}\w\pa_i\U\frac{y^i}{|y|}d\sigma\\
&+4(n-1)(n-2)\int_{\pa^+B_\rho^+}\U^{\frac{2n}{n-2}}V_i\frac{y^i}{|y|}d\sigma+\int_{\pa^+B_\rho^+}(\U^2\pa_k h_{ik}-\pa_k\U^2 h_{ik})\frac{y^i}{|y|}d\sigma\\
&+\int_{B_\rho^+}\frac{2}{n}(\U\Delta\U-\frac{n}{n-2}|\nabla \U|^2) {\rm tr}h\,dy-\int_{\pa^+ B_\rho^+}\U^2H_{ik}\pa_lH_{il}\frac{y^k}{|y|}d\sigma.
\end{align*}
Using \eqref{est:V} and the expression \eqref{eq:bdry_bubble} of $\U$, we estimate
\begin{align*}
\int_{\pa^+B_\rho^+}\pa_i\U\w\frac{y^i}{|y|}d\sigma\leq C(n,T_c) \sum_{a,b=1}^{n-1}\sum_{|\alpha|=1}^d|\pa^\alpha h_{ab}|\epsilon^{n-2}\rho^{|\alpha|+2-n},&\\
\int_{\pa^+B_\rho^+}\U^{\frac{2n}{n-2}}V_i\frac{y^i}{|y|}d\sigma\leq C(n,T_c)\sum_{a,b=1}^{n-1}\sum_{|\alpha|=1}^d|\pa^\alpha h_{ab}|\epsilon^n\rho^{|\alpha|-n},&\\
\int_{B_\rho^+}(\U\Delta\U-\frac{n}{n-2}|\nabla \U|^2) {\rm tr}h\,dy\leq C(n,T_c)\epsilon^{n-2}\rho^{2d+4-n}&
\end{align*}
\iffalse
\textcolor{blue}{Indeed we have
$$\int_{B_\rho^+}\U\Delta\U {\rm tr} h dx=\begin{cases}O(\epsilon^n \log\frac{\rho}{\epsilon}) &\hbox{~~if~~} n \hbox{~~is even,}\\
O(\epsilon^{n-1})  &\hbox{~~if~~} n \hbox{~~is odd.}
\end{cases}$$}
\fi
and use $|\pa H_{ij}|\leq C$ to show
\begin{align*}
\int_{\partial^+B_\rho^+}\U^2H_{ik}\pa_lH_{il}\frac{y^k}{|y|}d\sigma\leq C(n,T_c) \sum_{a,b=1}^{n-1}\sum_{|\alpha|=1}^d |\pa^\alpha h_{ab}|\epsilon^{n-2}\rho^{|\alpha|+3-n}.
\end{align*}
Hence combining the above estimates together, we obtain
\begin{align}
\int_{B_\rho^+}J_1dy\leq&-\frac{8(n-1)}{n-2}\int_{D_\rho}\pa_n\U \w d\sigma+\int_{\pa^+B_\rho^+}(\U^2\pa_k h_{ik}-\pa_k\U^2 h_{ik})\frac{y^i}{|y|}d\sigma\nonumber\\
&+C\sum_{a,b=1}^{n-1}\sum_{|\alpha|=1}^d|\pa^\alpha h_{ab}|\epsilon^{n-2}\rho^{|\alpha|+2-n}+C\rho^{2d+4-n}\epsilon^{n-2}.\label{est:J_1}
\end{align}

For $J_2$, by Proposition \ref{prop:second_variation_identity} and \eqref{eq:U:T:1} we have
$$J_2=-\frac{1}{4}Q_{ik,l}Q_{ik,l}-2\U^{\frac{2n}{n-2}}T_{ik}T_{ik}+\mathrm{div}\,\xi.$$
By \eqref{eq:xi} a direct computation yields
\begin{align*}
\int_{\pa^+B_\rho^+}\xi_i\frac{y^i}{|y|}d\sigma\leq C\sum_{a,b=1}^{n-1}\sum_{|\alpha|=1}^d|\pa^\alpha h_{ab}|^2\rho^{2|\alpha|+2-n}\epsilon^{n-2}.
\end{align*}
From this and Proposition \ref{prop:Q_and_h} we estimate
\begin{align}\label{est:J_2}
&\int_{B_\rho^+}J_2dy\no\\
=&-\frac{1}{4}\int_{B_\rho^+}Q_{ik,l}Q_{ik,l}dy-2\int_{B_\rho^+}\U^{\frac{2n}{n-2}}T_{ik}T_{ik} dy+\int_{\pa^+B_\rho^+}\xi_i\frac{y^i}{|y|}d\sigma-\int_{D_{\rho}}\xi_nd\sigma\no\\
\leq&-\int_{D_\rho}\xi_nd\sigma-\frac{n^2}{2(n-1)(n-2)}\int_{D_\rho}\partial_n \U\U S_{nn}^2d\sigma\no\\
&-\frac{1}{4}\lambda^\ast \sum_{a,b=1}^{n-1}\sum_{|\alpha|=1}^d |\pa^\alpha h_{ab}|^2\epsilon^{n-2}\int_{B_\rho^+}(\epsilon+|y|)^{2|\alpha|+2-2n}dy\no\\
&+C\sum_{a,b=1}^{n-1}\sum_{|\alpha|=1}^d |\pa^\alpha h_{ab}|^2\rho^{2|\alpha|+2-n}\epsilon^{n-2}.
\end{align}

Observe that when $|y|$ is sufficiently small, there hold $|h|\leq C|y|$ and 
\begin{align}\label{est:expan_inverse_g_x_0}
|g_{x_0}^{ik}-\delta_{ik}|\leq& C|h|,\no\\
|g_{x_0}^{ik}-\delta_{ik}+H_{ik}|\leq& C|h|^2+O(|y|^{d+1})\leq C|h||y|+O(|y|^{d+1}),\no\\
|g_{x_0}^{ik}-\delta_{ik}+h_{ik}-\frac{1}{2}H_{il}H_{kl}|\leq& C|h|^3+O(|y|^{d+2})\leq C|h|^2|y|+O(|y|^{d+2}).
\end{align}
By Proposition \ref{prop:expan_scalar}, \eqref{est:expan_inverse_g_x_0} and Young's inequality, we can bound $J_3$ and $J_4$ by
\begin{align}\label{est:J_4}
&J_3+J_4\no\\
\leq& C(n,T_c)\epsilon^{n-2}\sum_{a,b=1}^{n-1}\sum_{|\alpha|=1}^d |\pa^\alpha h_{ab}|^2(\epsilon+|y|)^{2|\alpha|+3-2n}\no\\
&+C(n,T_c)\epsilon^{n-2}\sum_{a,b=1}^{n-1}\sum_{|\alpha|=1}^d |\pa^\alpha h_{ab}|(\epsilon+|y|)^{|\alpha|+d+3-2n}\no\\
&+C(n,T_c)\epsilon^{n-2}(\epsilon+|y|)^{2d+4-2n}\no\\
\leq& \frac{1}{2}\lambda^*\epsilon^{n-2}\sum_{a,b=1}^{n-1}\sum_{|\alpha|=1}^d |\pa^\alpha h_{ab}|^2(\epsilon+|y|)^{2|\alpha|+2-2n}+C\epsilon^{n-2}(\epsilon+|y|)^{2d+4-2n}.
\end{align}

Consequently, combining the above \eqref{eq:mean_curv_hot}, \eqref{est:J_1}-\eqref{est:J_4} and using the decomposition \eqref{eq:decom_energy_integrand},  we conclude that
\begin{align}\label{est:energy_inside_intermediate}
&\int_{B^+_\rho}\left[\tfrac{4(n-1)}{n-2}|\nabla(\U+\w)|_{g_{x_0}}^2+R_{g_{x_0}}(\U+\w)^2\right] dy+2(n-1)\int_{D_\rho}h_{g_{x_0}}(\U+\w)^2d\sigma\no\\
\leq &\frac{4(n-1)}{n-2}\int_{B_\rho^+}\left[|\nabla \U|^2+n(n+2)\U^{\frac{4}{n-2}}\w^2\right]dy\no\\
&+\int_{\pa^+B_\rho^+}(\U^2\pa_j h_{ij}-\pa_j\U^2 h_{ij})\frac{y^i}{|y|}d\sigma-\int_{D_\rho}\frac{8(n-1)}{n-2}\pa_n\U \w d\sigma\no\\
&-\int_{D_\rho}\xi_nd\sigma-\frac{n^2}{2(n-1)(n-2)}\int_{D_\rho}\partial_n \U\U S_{nn}^2d\sigma\no\\
&-\frac{1}{2}\lambda^* \sum_{a,b=1}^{n-1}\sum_{|\alpha|=1}^d |\pa^\alpha h_{ab}|^2\epsilon^{n-2}\int_{B_\rho^+}(\epsilon+|y|)^{2|\alpha|+2-2n}dy\no\\
&+C\sum_{a,b=1}^{n-1}\sum_{|\alpha|=1}^d |\pa^\alpha h_{ab}|\epsilon^{n-2} \rho^{|\alpha|+2-n}+C\epsilon^{n-2}\rho^{2d+4-n}.
\end{align}
Testing equation \eqref{prob:half-space} with $\U$ and integrating over $B_\rho^+$, via integration by parts we obtain
\begin{align*}
&\frac{4(n-1)}{n-2}\int_{B_\rho^+}\left[|\nabla \U|^2+n(n+2)\U^{\frac{4}{n-2}}\w^2\right]dy\\
=&4n(n-1)\int_{B_\rho^+}\U^{\frac{4}{n-2}}\left(\U^2+\frac{n+2}{n-2}\w^2\right)dy\\
&+\frac{4(n-1)}{n-2}\int_{\pa^+B_\rho^+}\U\pa_i \U \frac{y^i}{|y|} d\sigma-4(n-1)T_c\int_{D_\rho}\U^{\frac{2(n-1)}{n-2}}d\sigma.
\end{align*}
Therefore, plugging this and \eqref{rel:bdr_deri_xi} into \eqref{est:energy_inside_intermediate} as well as using \eqref{prob:half-space} again, we obtain the desired assertion.
\end{proof}

\begin{proposition}\label{prop:volume_a} There exists some sufficiently small $\rho_0$ such that
\begin{align*}
&4n(n-1)\int_{B^+_\rho}\U^{\frac{4}{n-2}}\Big(\U^2+\frac{n+2}{n-2}\w^2\Big)dy\\
\leq&a Y_{a,b}(\mathbb{R}^n_+,\mathbb{R}^{n-1})\left(\int_{B_\rho^+}(\U+\w)^{\frac{2n}{n-2}}dy\right)^{\frac{n-2}{n}}+C\epsilon^n\sum_{a,b=1}^{n-1}\sum_{|\alpha|=1}^d |\pa^\alpha h_{ab}|\rho^{|\alpha|-n}\\
&+C\epsilon^n\sum_{a,b=1}^{n-1}\sum_{|\alpha|=1}^d |\pa^\alpha h_{ab}|^2\int_{B_\rho^+}(\epsilon+|y|)^{2|\alpha|+1-2n}dy
\end{align*}
for all $0<2\epsilon\leq \rho\leq \rho_0$.
\end{proposition}
\begin{proof} 
Notice that \eqref{est:w} gives $|\w|\leq C (\epsilon+|y|)\U$ in $B_{2\rho}^+$. By Lemma \ref{lem:Y_{a,b}halfspace}, we get
$$aY_{a,b}(\mathbb{R}_+^n,\mathbb{R}^{n-1})=4n(n-1)\left(\int_{\mathbb{R}_+^n}\U^{\frac{2n}{n-2}}dx\right)^{\frac{2}{n}}.$$
Together with the fact that $V_n=0$ on $D_{2\rho}$, the desired estimate can follow the same lines in \cite[Propositions 14-15]{Brendle2}.
\end{proof}
\begin{proposition}\label{prop:volume_b}
There exists some sufficiently small $\rho_0$ such that
\begin{align*}
&-4(n-1)T_c\int_{D_\rho}\U^{\frac{2}{n-2}}\Big(\U^2+2\U\w+\frac{n}{n-2}\w^2-\frac{n-2}{8(n-1)^2}\U^2S_{nn}^2\Big)d\sigma\\
\leq &2(n-1)bY_{a,b}(\mathbb{R}^n_+,\mathbb{R}^{n-1})\left(\int_{D_\rho}(\U+\w)^{\frac{2(n-1)}{n-2}}d\sigma\right)^{\frac{n-2}{n-1}}\\
&+C\sum_{a,b=1}^{n-1}\sum_{|\alpha|=1}^d |\pa^\alpha h_{ab}|\rho^{|\alpha|+1-n}\epsilon^{n-1}\\
&+C\sum_{a,b=1}^{n-1}\sum_{|\alpha|=1}^d |\pa^\alpha h_{ab}|^2\epsilon^{n-1}\rho\int_{D_\rho}(\epsilon+|y|)^{2|\alpha|+2-2n}d\sigma
\end{align*}
for all $0<2\epsilon\leq \rho< \rho_0$.
\end{proposition}
\begin{proof}
Since  \eqref{est:w} and \eqref{est:V} give $|\w|\leq C (\epsilon+|y|)\U$ and $|S_{nn}|\leq C(\epsilon+|y|)$ in $B_{2\rho}^+$, this assertion can follow the same lines in \cite[Proposition 8]{ChenSophie} (see also \cite[(3.23)]{almaraz5}) by using
$$-2T_c\left(\int_{\mathbb{R}^{n-1}}\U^{\frac{2(n-1)}{n-2}}d\sigma\right)^{\frac{1}{n-1}}=bY_{a,b}(\mathbb{R}^n_+,\mathbb{R}^{n-1})$$
in Lemma \ref{lem:Y_{a,b}halfspace}.
\end{proof}

For simplicity, we denote by $\Omega_\rho:=\Psi_{x_0}(B_\rho^+)$ the coordinate ball of radius $\rho$ under the Fermi coordinates around $x_0$.

\begin{lemma}\label{lem:ring1}
If $0<\epsilon\ll\rho<\rho_0$ for some sufficiently small $\rho_0$, in $M\backslash \Omega_\rho$  there holds
\begin{align*}
&|\ubar-\epsilon^{\frac{n-2}{2}}G|\\
\leq& C\sum_{a,b=1}^{n-1}\sum_{|\alpha|=1}^d |\pa^\alpha h_{ab}|\rho^{|\alpha|+2-n}\epsilon^{\frac{n-2}{2}}+C\rho^{d+3-n}|\log \rho|\epsilon^{\frac{n-2}{2}}+C\rho^{1-n}\epsilon^{\frac{n}{2}}.
\end{align*}
\end{lemma}
\begin{proof}
For $x\in M\backslash\Omega_\rho$, let $y=\Psi_{x_0}^{-1}(x) \in \mathbb{R}^n_+\setminus B_\rho^+$.  In $M\setminus \Omega_\rho$, we have
\begin{equation}\label{est:U-G}
\ubar(x)-\epsilon^{\frac{n-2}{2}}G(x)=\chi_{\rho}(y)\left[\U(y)+\w(y)-\epsilon^\frac{n-2}{2}G(\Psi_{x_0}(y))\right].
\end{equation}
Notice that
\begin{align*}
&\U(y)-\epsilon^{\frac{n-2}{2}}|y|^{2-n}\\
=&\epsilon^{\frac{n-2}{2}}|y|^{2-n}\left[\Big(1+\frac{(1+T_c^2)\epsilon^2}{|y|^2}-\frac{2y^nT_c \epsilon}{|y|^2}\Big)^{\frac{2-n}{2}}-1\right]\\
=&(n-2)y^n|y|^{-n}T_c\epsilon^{\frac{n}{2}}+O(\epsilon^{\frac{n+2}{2}}|y|^{-n}),
\end{align*}
then it yields
\begin{equation}\label{est:bubble_G}
|\U-\epsilon^{\frac{n-2}{2}}|y|^{2-n}|\leq C \epsilon^{\frac{n}{2}}\rho^{1-n} \hbox{~~in~~} B_{2\rho}^+\backslash B_\rho^+.
\end{equation}
Combining \eqref{est:bubble_G}, \eqref{est:Green_funct} and \eqref{est:w}, in $M\backslash\Omega_{\rho}$ we obtain 
\begin{align*}
&|\ubar-\epsilon^{\frac{n-2}{2}}G|\\
\leq&|\U-\epsilon^{\frac{n-2}{2}}|y|^{2-n}|+\e^{\frac{n-2}{2}}|G-|y|^{2-n}|+|\w|\\
\leq&C\sum_{a,b=1}^{n-1}\sum_{|\alpha|=1}^d |\pa^\alpha h_{ab}|\rho^{|\alpha|+2-n}\epsilon^{\frac{n-2}{2}}+\underline{C\rho^{d+3-n}|\log \rho|\epsilon^{\frac{n-2}{2}}}+C\rho^{1-n}\epsilon^{\frac{n}{2}},
\end{align*}
\footnote{In view of \eqref{est:Green_funct}, the underlined term can be improved to $C\rho^{d+3-n}\epsilon^{\frac{n-2}{2}}$ when $n \geq 5$ and $C|\log \rho|$ when $n=3,4$. Since this rough estimate goes through in the later part, we adopt it just for simplicity.}when $\epsilon\ll\rho<\rho_0$.
\end{proof}

\begin{lemma}\label{lem:ring2}
If $0<\epsilon\ll\rho<\rho_0$ for some sufficiently small $\rho_0$,
in $M\backslash\Omega_\rho$ there holds
\begin{align*}
&\rho^2\left|\frac{4(n-1)}{n-2}\Delta_{g_{x_0}}\ubar-R_{g_{x_0}}\ubar\right|\\
\leq& C\sum_{a,b=1}^{n-1}\sum_{|\alpha|=1}^d |\pa^\alpha h_{ab}|\rho^{|\alpha|+2-n}\epsilon^{\frac{n-2}{2}}+C\rho^{d+3-n}\epsilon^{\frac{n-2}{2}}+C\rho^{1-n}\epsilon^{\frac{n}{2}}.
\end{align*}
\end{lemma}

\begin{proof}Notice that $\ubar=\epsilon^\frac{n-2}{2}G$ in $M\backslash\Omega_{2\rho}$, the estimate is trivial by definition of $G$. Then it suffices to estimate the above inequality in $\Omega_{2\rho}\backslash\Omega_{\rho}$. To see this, by \eqref{est:U-G} we have
\begin{align*}
&\Delta_{g_{x_0}}\ubar-\frac{n-2}{4(n-1)}R_{g_{x_0}}\ubar\\
=&(\Delta_{g_{x_0}}\chi_{\rho})(\U+\w-\epsilon^{\frac{n-2}{2}}|y|^{2-n})+2\langle \nabla \chi_\rho,\nabla(\U+\w-\epsilon^{\frac{n-2}{2}}|y|^{2-n})\rangle_{g_{x_0}}\\
&-(\Delta_{g_{x_0}}\chi_{\rho})\epsilon^{\frac{n-2}{2}}(G-|x|^{2-n})-2\epsilon^{\frac{n-2}{2}}\langle\nabla \chi_\rho,\nabla(G-|x|^{2-n})\rangle_{g_{x_0}}\\
&+\chi_\rho\left[\Delta_{g_{x_0}}(\U+\w)-\frac{n-2}{4(n-1)}R_{g_{x_0}}(\U+\w)\right]\\
=&I_1+I_2+I_3,
\end{align*}
where $I_i (i=1,2,3)$ denotes the quantity in each corresponding line. By using \eqref{est:bubble_G} and $|\rho^2\Delta_{g_{x_0}}\chi_\rho|+|\rho\nabla \chi_{\rho}|_{g_{x_0}}\leq C$, we get
$$\rho^2I_1\leq  C\sum_{a,b=1}^{n-1}\sum_{|\alpha|=1}^d |\pa^\alpha h_{ab}|\rho^{|\alpha|+2-n}\epsilon^{\frac{n-2}{2}}+C\rho^{1-n}\epsilon^{\frac{n}{2}}.$$
Similarly \eqref{est:Green_funct} implies
$$\rho^2I_2\leq C\sum_{a,b=1}^{n-1}\sum_{|\alpha|=1}^d |\pa^\alpha h_{ab}|\rho^{|\alpha|+2-n}\epsilon^{\frac{n-2}{2}}+C\rho^{d+3-n}\epsilon^{\frac{n-2}{2}}+C\rho^{1-n}\epsilon^{\frac{n}{2}}.$$
For $I_3$, applying the property \eqref{est:w} of $\w$ and Proposition \ref{prop:expan_scalar}, we get
\begin{align*}
|R_{g_{x_0}}(\U+\w)|\leq C\sum_{a,b=1}^{n-1}\sum_{|\alpha|=1}^d |\pa^\alpha h_{ab}|\rho^{|\alpha|-n}\epsilon^{\frac{n-2}{2}}+C\rho^{d+1-n}\epsilon^{\frac{n-2}{2}}
\end{align*}
and 
\begin{align*}
&|\Delta_{g_{x_0}}(\U+\w)|\\
\leq& |(\Delta_{g_{x_0}}-\Delta_{\mathbb{R}^n})(\U+\w)|+ C\epsilon^{\frac{n+2}{2}}\rho^{-n-2}+C\epsilon^{\frac{n-2}{2}}\sum_{a,b=1}^{n-1}\sum_{|\alpha|=1}^d |\pa^\alpha h_{ab}|\rho^{|\alpha|-n}\\
\leq&C\sum_{a,b=1}^{n-1}\sum_{|\alpha|=1}^d |\pa^\alpha h_{ab}|\rho^{|\alpha|-n}\epsilon^{\frac{n-2}{2}}+C\rho^{d+1-n}\epsilon^{\frac{n-2}{2}}+C\rho^{-n-2}\epsilon^{\frac{n+2}{2}}
\end{align*}
Therefore 
$$\rho^2I_3\leq C\sum_{a,b=1}^{n-1}\sum_{|\alpha|=1}^d |\pa^\alpha h_{ab}|\rho^{|\alpha|+2-n}\epsilon^{\frac{n-2}{2}}+C\rho^{d+3-n}\epsilon^{\frac{n-2}{2}}+C\rho^{-n}\epsilon^{\frac{n+2}{2}}.$$
Collecting all the above estimates on $I_1$-$I_3$, we get the desired assertion.
\end{proof}

We now arrive at the key Proposition \ref{prop:energy_est}.
\begin{proposition}\label{prop:energy_est} 
If $0<\e\ll\rho<\rho_0$ for some sufficiently small $\rho_0$, there holds
\begin{align*}
&\int_M\left[\tfrac{4(n-1)}{n-2}|\nabla \ubar|_{g_{x_0}}^2+R_{g_{x_0}}\ubar^2\right]d\mu_{g_{x_0}}+2(n-1)\int_{\partial M}h_{g_{x_0}}\ubar^2d\sigma_{g_{x_0}}\\
\leq &Y_{a,b}(\mathbb{R}^n,\mathbb{R}^{n-1})\left[a\left(\int_{M}\ubar^\frac{2n}{n-2}d\mu_{g_{x_0}}\right)^{\frac{n-2}{n}}+2(n-1)b\left(\int_{\partial M}\ubar^{\frac{2(n-1)}{n-2}}d\sigma_{g_{x_0}}\right)^{\frac{n-2}{n-1}}\right]\\
&-\epsilon^{n-2}\mathcal{I}(x_0,\rho)-\frac{1}{C}\eta_{\mathcal{Z}^c}(x_0)\lambda^* \epsilon^{n-2}\int_{B_\rho^+}|W_{g_0}(y)|_{g_0}^2(\epsilon+|y|)^{6-2n}dy\\
&-\frac{1}{C}\eta_{\mathcal{Z}^c}(x_0)\lambda^*\epsilon^{n-2}\int_{D_\rho}|\mathring{\pi}_{g_0}(y)|_{g_0}^2(\epsilon+|y|)^{5-2n}d\sigma+C\rho^{2d+4-n}|\log \rho|^2 \epsilon^{n-2}\\
&+C\left(\frac{\epsilon}{\rho}\right)^{n-2}\frac{1}{\log(\rho/\epsilon)}+C^\ast \left(\frac{\epsilon}{\rho}\right)^{n-1},
\end{align*}
where $\eta_{\mathcal{Z}^c}$ is the characteristic function of $\mathcal{Z}^c=\pa M\setminus \mathcal{Z}$ defined on $\pa M$ and $C, C^\ast$ depend on $n,g_0,T_c,\rho_0$.
\end{proposition}

\begin{proof}
Observe that
\begin{align*}
&\int_{M\backslash\Omega_\rho}\left[\tfrac{4(n-1)}{n-2}|\nabla \ubar|_{g_{x_0}}^2+R_{g_{x_0}}\ubar^2\right]d\mu_{g_{x_0}}+2(n-1)\int_{\partial M\backslash \Omega_\rho}h_{g_{x_0}}\ubar^2 d\sigma_{g_{x_0}}\\
=&\int_{M\backslash\Omega_\rho}\left(-\frac{4(n-1)}{n-2}\Delta_{g_{x_0}}\ubar+R_{g_{x_0}}\ubar^2\right)(\ubar-\epsilon^{\frac{n-2}{2}}G)d\mu_{g_{x_0}}\\
&+\frac{4(n-1)}{n-2}\int_{\partial(M\backslash\Omega_\rho)}\left[\frac{\partial\ubar}{\partial{\nu_{g_{x_0}}}}\ubar+\epsilon^{\frac{n-2}{2}}\left(\ubar\frac{\partial G}{\partial\nu_{g_{x_0}} }-G\frac{\partial \ubar}{\partial\nu_{g_{x_0}}}\right)\right]d\sigma_{g_{x_0}}\\
&+2(n-1)\int_{\partial M\backslash \Omega_\rho}h_{g_{x_0}}\ubar^2d\sigma_{g_{x_0}}\\
=&II_1+II_2+II_3,
\end{align*}
where $II_i(i=1,2,3)$ denotes the quantity in each corresponding line on the right hand side of the first identity. By Lemmas \ref{lem:ring1} and \ref{lem:ring2}, we get
\begin{align*}
&\sup_{M\backslash\Omega_\rho}\left[|\ubar-\epsilon^{\frac{n-2}{2}}G|+\rho^2\left|\frac{4(n-1)}{n-2}\Delta_{g_{x_0}}\ubar-R_{g_{x_0}}\ubar\right|\right]\\
\leq&C\sum_{a,b=1}^{n-1}\sum_{|\alpha|=1}^d |\pa^\alpha h_{ab}| \rho^{|\alpha|+2-n}\epsilon^{\frac{n-2}{2}}+C\rho^{d+3-n}|\log \rho|\epsilon^{\frac{n-2}{2}}+C\rho^{1-n}\epsilon^{\frac{n}{2}}.
\end{align*}
From this, one can estimate $II_1$ as
\begin{align}\label{est:I_1}
II_1=&\int_{\Omega_{2\rho}\backslash\bar \Omega_\rho}\left(-\tfrac{4(n-1)}{n-2}\Delta_{g_{x_0}}\ubar+R_{g_{x_0}}\ubar\right)(\ubar-\epsilon^{\frac{n-2}{2}}G)d\mu_{g_{x_0}}\no\\
\leq &C\sum_{a,b=1}^{n-1}\sum_{|\alpha|=1}^d |\pa^\alpha h_{ab}|^2 \rho^{2|\alpha|+2-n}\epsilon^{n-2}+C\rho^{2d+4-n}|\log \rho|^2 \epsilon^{n-2}+C\rho^{-n}\epsilon^n.
\end{align}
For $II_2$, we divide the integral into two parts $II_2=II^{(1)}_2+II^{(2)}_2$ according to $\partial(M\backslash\Omega_\rho)=(\partial M\backslash \Omega_\rho)\cup(\partial \Omega_\rho\backslash\partial M)$. Namely $II^{(1)}_2$ is the integral over $\partial M\backslash \Omega_\rho$ while $II^{(2)}_2$ is over $\partial \Omega_\rho\backslash\partial M$. Let us deal with $II^{(1)}_2$ first. In $\partial M\backslash \Omega_\rho$, by Lemma \ref{lem:phi}, \eqref{prob:half-space}, \eqref{def:Green_funct} and \eqref{mean_curv_Fermi}, we have
\begin{align*}
&\sup_{\partial M\cap(\Omega_{2\rho}\backslash \bar \Omega_\rho)}\left|\frac{\partial\ubar}{\partial\nu_{g_{x_0}}}\right|\\
\leq& \sup_{\partial M\cap(\Omega_{2\rho}\backslash \bar \Omega_\rho)}\left[|\partial_n\U+\partial_n\w|+\e^{\frac{n-2}{2}}\Big|\frac{\pa G}{\pa \nu_{g_{x_0}}}\Big|\right]\\
\leq& C\epsilon^{\frac{n}{2}}\rho^{-n}+C\epsilon^{\frac{n}{2}}\sum_{a,b=1}^{n-1}\sum_{|\alpha|=1}^d |\pa^\alpha h_{ab}|\rho^{|\alpha|-n}+C\e^{\frac{n-2}{2}}\rho^{2d+3-n}.
\end{align*}
Using \eqref{def:Green_funct}, Lemma \ref{lem:ring1}, \eqref{mean_curv_Fermi}  and $\ubar=\e^{\frac{n-2}{2}}G$ in $M\setminus \Omega_{2\rho}$, we have
\begin{align}\label{est:I_2_1}
&II^{(1)}_2+II_3\nonumber\\
=&\frac{4(n-1)}{n-2}\int_{\partial M\backslash\bar{\Omega}_\rho}\left[\frac{\partial\ubar}{\partial{\nu_{g_{x_0}}}}\ubar+\epsilon^\frac{n-2}{2}\Big(\ubar\frac{\partial G}{\partial\nu_{g_{x_0}}}-G\frac{\partial \ubar}{\partial\nu_{g_{x_0}}}\Big)\right]d\sigma_{g_{x_0}}\no\\
&+II_3\nonumber\\
=&\frac{4(n-1)}{n-2}\int_{\partial M\cap(\Omega_{2\rho}\backslash\bar{\Omega}_\rho)}\frac{\partial \ubar}{\partial\nu_{g_{x_0}}}(\ubar-\epsilon^\frac{n-2}{2}G)d\sigma_{g_{x_0}}\no\\
&+2(n-1)\int_{\pa M \cap (\Omega_{2\rho}\setminus \bar \Omega_\rho)}h_{g_{x_0}}\ubar(\ubar-\e^{\frac{n-2}{2}}G) d\sigma_{g_{x_0}}\nonumber\\
\leq &C\sum_{a,b=1}^{n-1}\sum_{|\alpha|=1}^d |\pa^\alpha h_{ab}|^2\rho^{2|\alpha|+1-n}\epsilon^{n-1}+|\pa^\alpha h_{ab}|\rho^{|\alpha|+1-n}\epsilon^{n-1}\no\\
&+C\rho^{d+2-n}\epsilon^{n-1}+C\rho^{-n}\epsilon^n+C\rho^{2d+4-n}|\log \rho|\epsilon^{n-2}.
\end{align}
Next we start to estimate $II^{(2)}_2$ whose integral domain is $\partial \Omega_\rho\backslash \partial M$. It is not hard to verify that the outward unit normal $\nu_{g_{x_0}}$ on $\pa \Omega_\rho\backslash \pa M:=\Psi_{x_0}(\pa^+ B_\rho^+)$ is given by
$$\nu_{g_{x_0}}=\frac{g_{x_0}^{ik}y^k}{\|y\|} \pa_{y^i} \hbox{~~for~~} y \in \pa^+ B_\rho^+,$$
where $\|y\|^2:=g_{x_0}^{kl}y^ky^l\leq\rho^2(1+C|h|)$ on $\pa^+ B_\rho^+$. Since $\ubar=\U+\psi$ on $\pa^+ B_\rho^+$, from \eqref{est:expan_inverse_g_x_0} we estimate
\begin{align}\label{est:II_2^(2)_first}
&\int_{\partial \Omega_\rho\backslash \partial M}\frac{\partial \ubar}{\partial\nu_{g_{x_0}}}\ubar d\sigma_{g_{x_0}}\no\\
=&-\int_{\partial^+B_\rho^+}g^{ij}_{x_0}\partial_i\ubar \frac{y^j}{\|y\|}\ubar d\sigma+O(\rho^{2d+4-n}\epsilon^{n-2})\no\\
\leq&\int_{\partial^+B_\rho^+}\left(-\partial_i\ubar+\partial_j\ubar h_{ij}\right)\frac{y^i}{|y|}(1+C|h|)\ubar d\sigma\no\\
&+C\sum_{a,b=1}^{n-1}\sum_{|\alpha|=1}^d |\pa^\alpha h_{ab}|^2\rho^{2|\alpha|+2-n}\epsilon^{n-2}+O(\rho^{2d+4-n}\epsilon^{n-2})\no\\
\leq &\int_{\partial^+B_\rho^+}\left(-\partial_i\U +\partial_j\U h_{ij}\right)\frac{y^i}{|y|}\U d\sigma+C\sum_{a,b=1}^{n-1}\sum_{|\alpha|=1}^d |\pa^\alpha h_{ab}|\rho^{|\alpha|+2-n}\epsilon^{n-2}\no\\
&\quad+C\sum_{a,b=1}^{n-1}\sum_{|\alpha|=1}^d |\pa^\alpha h_{ab}|^2\rho^{2|\alpha|+2-n}\epsilon^{n-2}+C\rho^{2d+4-n}\epsilon^{n-2}.
% \leq& -\int_{\partial^+B_\rho^+}\partial_i\U\U\frac{x_i}{|x|}d\sigma+C\sum_{i,j=1}^n\sum_{1\leq |\alpha|\leq d}|h_{ij,\alpha}|\rho^{|\alpha|+2-n}\epsilon^{n-2}\\
% &\quad+C\sum_{i,j=1}^n\sum_{1\leq |\alpha|\leq d}|h_{ij,\alpha}|^2\rho^{2|\alpha|+2-n}\epsilon^{n-2}+C\rho^{2d+4-n}\epsilon^{n-2}
\end{align}
Similarly we have
\begin{align}\label{est:energy_outside}
&\epsilon^{\frac{n-2}{2}}\int_{\partial \Omega_\rho\backslash \partial M}\left(\ubar \frac{\partial G}{\partial\nu_{g_{x_0}}}-G\frac{\partial G}{\partial \nu_{g_{x_0}}}\right)d\sigma_{g_{x_0}}\no\\
\leq &-\epsilon^\frac{n-2}{2}\int_{\partial^+B_\rho^+}\left(\ubar \partial_iG-G\partial_i\ubar \right)\frac{y^i}{|y|}(1+C|h|)d\sigma\no\\
&+\epsilon^\frac{n-2}{2}\int_{\partial^+B_\rho^+}h_{ij}\frac{y^i}{|y|}(1+C|h|)\left(\ubar\partial_jG-G\partial_j\ubar\right)d\sigma\no\\
&+C\sum_{a,b=1}^{n-1}\sum_{|\alpha|=1}^d |\pa^\alpha h_{ab}|^2\rho^{2|\alpha|+2-n}\epsilon^{n-2}+C\rho^{2d+4-n}\epsilon^{n-2}\no\\
\leq &-\epsilon^\frac{n-2}{2}\int_{\partial^+B_\rho^+}\left(\U \partial_iG-G\partial_i\U \right)\frac{y^i}{|y|}d\sigma\no\\
&+\epsilon^\frac{n-2}{2}\int_{\partial^+B_\rho^+}h_{ij}\frac{y^i}{|y|}\left(\U\partial_jG-G\partial_j\U\right)d\sigma\no\\
&+C\sum_{a,b=1}^{n-1}\sum_{|\alpha|=1}^d |\pa^\alpha h_{ab}|\rho^{|\alpha|+2-n}\epsilon^{n-2}+C\sum_{a,b=1}^{n-1}\sum_{|\alpha|=1}^d |\pa^\alpha h_{ab}|^2\rho^{2|\alpha|+2-n}\epsilon^{n-2}\no\\
&+C\rho^{2d+4-n}\epsilon^{n-2}.
\end{align}
From \eqref{est:Green_funct} and \eqref{est:bubble_G}, on $\partial^+B_\rho^+$ we get
% \begin{align} (\textcolor{red}{double check})
% &|(\epsilon^2+|x-T_c\epsilon e_n|^2)^{-\frac{n}{2}}-|x|^{-n}|\\
% \leq& |(\epsilon^2+|x-T_c\epsilon|^2)^{-\frac{n}{2}}-|x-T_c\epsilon e_n|^{-n}|+||x-T_c\epsilon e_n|^{-n}-|x|^{-n}|\\
% \leq& C(n)\epsilon^2|x-T_c\epsilon e_n|^{-n-2}+C(T_c,n)\epsilon |x|^{-n-1}\leq C(T_c,n)\epsilon|x|^{-n}
% \end{align}
% on $\partial^+B_\rho^+$, therefore,
% \begin{align}
% &|\partial_i\U-\epsilon^\frac{n-2}{2}\partial_i|x|^{2-n}|\\
% =&(n-2)\epsilon^{\frac{n-2}{2}}|(\epsilon^2+|x_i-T_c\epsilon e_n|^2)^{-n/2}(x-T_c\epsilon e_n)-|x|^{-n}x_i|\\
% \leq &C\epsilon^{\frac{n}{2}}|x|^{-n}
% \end{align}
\begin{align*}
&\epsilon^{\frac{n-2}{2}}|\partial_i\U G-\partial_iG\U|\\
\leq& |\partial_i\U(\epsilon^{\frac{n-2}{2}}G-\U)|+|\U\partial_i(\epsilon^{\frac{n-2}{2}}G-\U)|\\
\leq&C\epsilon^{n-2}\sum_{a,b=1}^{n-1}\sum_{|\alpha|=1}^d |\pa^\alpha h_{ab}|\rho^{|\alpha|+3-2n}+C\epsilon^{n-2}\rho^{d+4-2n}|\log \rho|+C\epsilon^{n-1}\rho^{2-2n}
\end{align*}
and then
\begin{align}\label{est:RHS2}
&\epsilon^\frac{n-2}{2}\int_{\partial^+B_\rho^+}h_{ij}\frac{y^i}{|y|}\left(\U\partial_jG-G\partial_j\U\right)d\sigma\no\\
\leq &C\epsilon^{n-2}\sum_{a,b=1}^{n-1}\sum_{|\alpha|=1}^d |\pa^\alpha h_{ab}|^2\rho^{2|\alpha|+2-n}+C\rho^{2d+4-n}|\log \rho|\epsilon^{n-2}+C\epsilon^{n-1}\rho^{2-n}.
\end{align}
Hence plugging \eqref{est:RHS2} into \eqref{est:energy_outside}, we obtain
\begin{align}\label{est:II_2^(2)_second}
&\epsilon^{\frac{n-2}{2}}\int_{\partial \Omega_\rho\backslash \partial M}(\ubar \frac{\partial G}{\partial \nu_{g_{x_0}}}-G\frac{\partial G}{\partial \nu_{g_{x_0}}})d\sigma_{g_{x_0}}\no\\
\leq &-\epsilon^\frac{n-2}{2}\int_{\partial^+B_\rho^+}\left(\U \partial_iG-G\partial_i\U \right)\frac{y^i}{|y|}d\sigma+C\sum_{a,b=1}^{n-1}\sum_{|\alpha|=1}^d |\pa^\alpha h_{ab}|\rho^{|\alpha|+2-n}\epsilon^{n-2}\no\\
&+C\sum_{a,b=1}^{n-1}\sum_{|\alpha|=1}^d |\pa^\alpha h_{ab}|^2\rho^{2|\alpha|+2-n}\epsilon^{n-2}+C\rho^{2d+4-n}|\log \rho|\epsilon^{n-2}+C\epsilon^{n-1}\rho^{2-n}.
\end{align}
Consequently combining \eqref{est:II_2^(2)_first} and \eqref{est:II_2^(2)_second}, we can get 
\begin{align}\label{est:I_2_2}
II^{(2)}_2\leq& -\tfrac{4(n-1)}{n-2}\int_{\partial^+B_\rho^+}\left[\partial_i\U\U-\partial_j\U\U h_{ij}+\epsilon^\frac{n-2}{2}\left(\U \partial_iG-G\partial_i\U \right)\right]\frac{y^i}{|y|}d\sigma\no\\
&+C\sum_{a,b=1}^{n-1}\sum_{|\alpha|=1}^d |\pa^\alpha h_{ab}|\rho^{|\alpha|+2-n}\epsilon^{n-2}+C\sum_{a,b=1}^{n-1}\sum_{|\alpha|=1}^d |\pa^\alpha h_{ab}|^2\rho^{2|\alpha|+2-n}\epsilon^{n-2}\no\\
&+C\rho^{2d+4-n}|\log \rho|\epsilon^{n-2}+C\epsilon^{n-1}\rho^{2-n}.
\end{align}
Therefore collecting the estimates \eqref{est:I_1} for $II_1$, \eqref{est:I_2_1} for $II_2^{(1)}+II_3$ and \eqref{est:I_2_2} for $II^{(2)}_2$ together, when $\e\ll\rho<\rho_0$ we  obtain
 \begin{align}\label{energy_outside_ball}
&\int_{M\backslash\Omega_\rho}\left[\tfrac{4(n-1)}{n-2}|\nabla\ubar|_{g_{x_0}}^2+R_{g_{x_0}}\ubar^2\right]d\mu_{g_{x_0}}+2(n-1)\int_{\partial M\backslash\Omega_\rho}h_{g_{x_0}}\ubar^2 d\sigma_{g_{x_0}}\no\\
\leq &\frac{4(n-1)}{n-2}\int_{\partial^+B_\rho^+}\left[-\partial_i\U\U+\partial_j\U\U h_{ij}-\epsilon^{\frac{n-2}{2}}(\U\partial_iG-G\partial_i\U)\right]\frac{y^i}{|y|} d\sigma\no\\
&+C\sum_{a,b=1}^{n-1}\sum_{|\alpha|=1}^d |\pa^\alpha h_{ab}|\rho^{|\alpha|+2-n}\epsilon^{n-2}+C\sum_{a,b=1}^{n-1}\sum_{|\alpha|=1}^d |\pa^\alpha h_{ab}|^2\rho^{2|\alpha|+2-n}\epsilon^{n-2}\no\\
&+C\rho^{2d+4-n}|\log \rho|^2\epsilon^{n-2}+C\rho^{2-n}\epsilon^{n-1}.
\end{align}

Finally, since  $d\mu_{g_{x_0}}=\big(1+O(|y|^{2d+2})\big)dy$ and $d\sigma_{g_{x_0}}=\big(1+O(|y|^{2d+2})\big)d\sigma$ under the Fermi coordinates around $x_0\in \pa M$, noticing that Propositions \ref{prop:near_pole}-\ref{prop:volume_b} and \eqref{energy_outside_ball} give the estimates of energy $E[\ubar]$ in the interior of $B_\rho^+=\Psi_{x_0}^{-1}(\Omega_\rho)$ and in the exterior of $\Omega_\rho$, respectively, we conclude that
 \begin{align}\label{est:total_energy}
&\int_{M}\left[\tfrac{4(n-1)}{n-2}|\nabla\ubar|_{g_{x_0}}^2+R_{g_{x_0}}\ubar^2\right]d\mu_{g_{x_0}}+2(n-1)\int_{\partial M}h_{g_{x_0}}\ubar^2d\sigma_{g_{x_0}}\no\\
\leq &Y_{a,b}(\mathbb{R}^n,\mathbb{R}^{n-1})\left[a\left(\int_{M}\ubar^\frac{2n}{n-2}d\mu_{g_{x_0}}\right)^{\frac{n-2}{n}}+2(n-1)b\left(\int_{\partial M}\ubar^{\frac{2(n-1)}{n-2}}d\sigma_{g_{x_0}}\right)^{\frac{n-2}{n-1}}\right]\no\\
&+\int_{\pa^+B_\rho^+}(\U^2\pa_j h_{ij}+\frac{n}{n-2}\pa_j\U^2 h_{ij})\frac{y^i}{|y|}d\sigma\no\\
&-\frac{4(n-1)}{n-2}\epsilon^{\frac{n-2}{2}}\int_{\partial^+ B_\rho^+}(\U\partial_iG-G\partial_i\U)\frac{y^i}{|y|}d\sigma\no\\
&-\frac{1}{4}\lambda^* \sum_{a,b=1}^{n-1}\sum_{|\alpha|=1}^d |\pa^\alpha h_{ab}|^2\epsilon^{n-2}\int_{B_\rho^+}(\epsilon+|y|)^{2|\alpha|+2-2n}dy\no\\
&+C\sum_{a,b=1}^{n-1}\sum_{|\alpha|=1}^d |\pa^\alpha h_{ab}|^2\epsilon^{n-1}\rho\int_{D_\rho}(\epsilon+|y|)^{2|\alpha|+2-2n}d\sigma\no\\
&+C\sum_{a,b=1}^{n-1}\sum_{|\alpha|=1}^d |\pa^\alpha h_{ab}|\rho^{|\alpha|+2-n}\epsilon^{n-2}+C\rho^{2d+4-n}|\log \rho|^2 \epsilon^{n-2}\no\\
&+C\epsilon^{n-1}\rho^{2-n},
\end{align}
where we have used the following estimate:
\begin{align*}
&\epsilon^n\int_{B_\rho^+}(\epsilon+|y|)^{2|\alpha|+1-2n}dy\\
\leq& C\epsilon^{n-1}\int_{B_\rho^+}(\epsilon+|y|)^{2|\alpha|+2-2n}dy
\leq\frac{\lambda^\ast}{4} \epsilon^{n-2}\int_{B_\rho^+}(\epsilon+|y|)^{2|\alpha|+2-2n}dy
\end{align*}
by choosing $\e\ll\rho<\rho_0$.
By \eqref{est:Green_funct} and the expression \eqref{eq:bdry_bubble} of $\U$, we get
\begin{align}\label{est:energy_mass}
&\int_{\pa^+B_\rho^+}(\U^2\pa_j h_{ij}+\frac{n}{n-2}\pa_j \U^2 h_{ij})\frac{y^i}{|y|}d\sigma\no\\
&-\frac{4(n-1)}{n-2}\epsilon^{\frac{n-2}{2}}\int_{\partial^+ B_\rho^+}(\U\partial_iG-G\partial_i\U)\frac{y^i}{|y|}d\sigma\no\\
\leq& -\epsilon^{n-2}\mathcal{I}(x_0,\rho)+C\sum_{a,b=1}^{n-1}\sum_{|\alpha|=1}^d |\pa^\alpha h_{ab}|\rho^{|\alpha|+1-n}\epsilon^{n-1}+C\epsilon^{n-1}\rho^{1-n}.
\end{align}
Notice that 
$$|W_{g_0}(x)|_{g_0}=f_{x_0}^{\frac{4}{n-2}}|W_{g_{x_0}}(x)|_{g_{x_0}}\leq C |\pa^2 h|+C|\pa h| \quad\hbox{~~in~~} M$$
and
$$|\mathring{\pi}_{g_0}(x)|_{g_0}=f_{x_0}^{\frac{2}{n-2}}|W_{g_{x_0}}(x)|_{g_{x_0}}\leq C |\pa h| \quad\hbox{~~on~~}\pa M.$$
By choosing $\rho_0$ small enough with all $\rho<\rho_0$, it is not hard to show that
$$C\epsilon^{n-1}\rho\int_{D_\rho}(\epsilon+|y|)^{2|\alpha|+2-2n}d\sigma\leq \frac{\lambda^\ast}{8}\epsilon^{n-2}\int_{B_\rho^+}(\epsilon+|y|)^{2|\alpha|+2-2n}dy.$$
Recall that we define by $\mathcal{Z}$ the set of all points $x_0\in \partial M$ satisfying
\begin{align*}
\limsup_{x\to x_0} d_{g_0}(x,x_0)^{2-d}|W_{g_0}(x)|_{g_0}=\limsup_{x\to x_0} d_{g_0}(x,x_0)^{1-d}|\mathring{\pi}_{g_0}(x)|_{g_0}=0.
\end{align*}
From these estimates, \eqref{est:total_energy} and \eqref{est:energy_mass}, a similar argument in \cite[Corollary 3.10]{almaraz5} yields
 \begin{align*}
&\int_{M}\left[\tfrac{4(n-1)}{n-2}|\nabla\ubar|_{g_{x_0}}^2+R_{g_{x_0}}\ubar^2\right]d\mu_{g_{x_0}}+2(n-1)\int_{\partial M}h_{g_{x_0}}\ubar^2d\sigma_{g_{x_0}}\\
\leq &Y_{a,b}(\mathbb{R}^n,\mathbb{R}^{n-1})\left[a\left(\int_{M}\ubar^\frac{2n}{n-2}d\mu_{g_{x_0}}\right)^{\frac{n-2}{n}}+2(n-1)b\left(\int_{\partial M}\ubar^{\frac{2(n-1)}{n-2}}d\sigma_{g_{x_0}}\right)^{\frac{n-2}{n-1}}\right]\\
&-\epsilon^{n-2}\mathcal{I}(x_0,\rho)-\frac{1}{C}\eta_{\mathcal{Z}^c}(x_0)\lambda^* \epsilon^{n-2}\int_{B_\rho^+}|W_{g_0}(y)|_{g_0}^2(\epsilon+|y|)^{6-2n}dy\\
&-\frac{1}{C}\eta_{\mathcal{Z}^c}(x_0)\lambda^*\epsilon^{n-2}\int_{D_\rho}|\mathring{\pi}_{g_0}(y)|_{g_0}^2(\epsilon+|y|)^{5-2n}d\sigma+C^\ast\rho^{2d+4-n}|\log \rho|^2 \epsilon^{n-2}\\
&+C\left(\frac{\epsilon}{\rho}\right)^{n-2}\frac{1}{\log(\rho/\epsilon)}+C\left(\frac{\epsilon}{\rho}\right)^{n-1},
\end{align*}
by recalling that $\eta_{\mathcal{Z}^c}$ is the characteristic function of $\mathcal{Z}^c=\pa M\setminus \mathcal{Z}$.
\end{proof}

Next we describe the continuity of $\mathcal{I}(x_0,\rho)$ over $\mathcal{Z}$ as in \cite[Proposition 3.11]{almaraz5} and some characterization of its limit as $\rho \to 0$ in \cite[Proposition 3.12]{almaraz5}, see also \cite[Proposition 4.3]{Brendle-Chen}.  We restate them here for convenience.
\begin{proposition}\label{prop:continuity_I(x_0,rho)}
The functions $\mathcal{I}(x_0,\rho)$ converge to a continuous function $\mathcal{I}(x_0):\mathcal{Z}\to\mathbb{R}$~ uniformly for all $x_0 \in \mathcal{Z}$, as $\rho\to 0$.
\end{proposition}
\begin{proposition}\label{propo19} 
Let $x_0\in\mathcal{Z}$ and consider inverted coordinates $\Phi: y \in \bar M \setminus\{x_0\} \mapsto z:=y/|y|^2$, where $y=(y^1,\cdots,y^n)$ are Fermi coordinates centered at $x_0$. If we define the metric $\bar{g}_{x_0}=\Phi_\ast(G_{x_0}^{4/(n-2)}g_{x_0})$ on $\bar M\backslash\{x_0\}$, then the following statements hold:

(i) $(\bar M\backslash\{x_0\},\bar{g}_{x_0})$ is an asymptotically flat manifold with order  $d+1>\frac{n-2}{2}$ (in the sense of Definition \ref{def:asym}), and  satisfies $R_{\bar{g}_{x_0}}\equiv 0$ and $h_{\bar{g}_{x_0}}\equiv 0$.

(ii) We have 
\begin{align*}
&\mathcal{I}(x_0)=\lim_{R\to\infty}
\left[\int_{\pa^+B^+_R}\frac{z^i}{|z|}\pa_{z^j}
\bar{g}_{x_0}(\pa_{z^i},\pa_{z^j})d\sigma
-\int_{\pa^+B^+_R}\frac{z^i}{|z|}\pa_{z^i}\bar{g}_{x_0}
(\pa_{z^j},\pa_{z^j})d\sigma\right].
\end{align*}
In particular, $\mathcal{I}(x_0)$ is the mass $m(\bar{g}_{x_0})$ of $(\bar M\backslash\{x_0\},\bar{g}_{x_0})$.
\end{proposition}

\begin{proof}[Proof of Theorem \ref{Thm:main}]
(i) When $\pa M \setminus \mathcal{Z}\neq \emptyset$, we choose $x_0 \in \pa M \setminus \mathcal{Z}$. Thus the desired assertion follows from Proposition \ref{prop:energy_est}.

(ii) Assume that $\mathcal{I}(x_0)>0$ for some $x_0\in \mathcal{Z}$, it follows from Proposition \ref{prop:continuity_I(x_0,rho)} that
$$\mathcal{I}(x_0,\rho)>C^\ast \rho^{2d+4-n}|\log \rho|^2$$
for all $0<\rho<\rho_0$, where $\rho_0, C^\ast$ are the positive constants in Proposition \ref{prop:energy_est}. Based on the key estimate in Proposition \ref{prop:energy_est}, Theorem \ref{Thm:main} follows the same lines of \cite[Proposition 3.7]{almaraz5}.
\end{proof}

\end{document}